\documentclass[reqno,oneside,11pt]{amsart}

\usepackage[a4paper,text={14.5cm,23cm},centering,headheight=15.5pt]{geometry}

\usepackage[utf8]{inputenc}
\usepackage[T1]{fontenc}
\usepackage[svgnames,x11names]{xcolor}
\usepackage{lmodern}
\usepackage{microtype} %better text justification
\usepackage{amssymb} %\setminus,leqslant... 
\usepackage{amsmath}
\usepackage{amsthm}
\usepackage{mathtools} % multlined environment

\usepackage{booktabs, tabularx}
\AtBeginEnvironment{tabular}{
\newcolumntype{L}{>{$}l<{$}}  
\newcolumntype{R}{>{$}r<{$}}
\newcolumntype{C}{>{$}c<{$}}
}

\usepackage[shortlabels]{enumitem}
\setlist{nosep, itemsep=1pt}
\setlist[1]{wide}
\setlist[2]{leftmargin=15mm}
\setlist[enumerate]{label=\rm{(\arabic*)}}
\setlist[enumerate,2]{label=\rm({\roman*}), }
\setlist[itemize]{label=\raisebox{0.25ex}{\tiny$\bullet$}}

\usepackage{tikz}
\usetikzlibrary{cd,
                arrows,
                positioning,
                calc,
                decorations.markings,
                backgrounds,
                shapes.geometric,
                math,
                intersections}
\tikzset{>=stealth}
\tikzset{commutative diagrams/row sep/normal=0.35cm}
\tikzset{map/.style={row sep=0em, column sep=0em}}
\tikzset{bullet/.style={circle,fill=black,minimum size=5pt,inner sep=0pt}}
\tikzset{graph/.style={circle,thick,draw,inner sep=0mm,minimum size=7mm,fill=gray!20}}

\tikzcdset{arrow style=tikz}

\DeclareFontFamily{U}{mathx}{\hyphenchar\font45}
\DeclareFontShape{U}{mathx}{m}{n}{<-> mathx10}{}
\DeclareSymbolFont{mathx}{U}{mathx}{m}{n}

\DeclareMathAccent{\widebar}{0}{mathx}{"73}

\renewcommand{\to}{ \, \tikz[baseline=-.6ex] \draw[->,line width=.5] (0,0) -- +(.5,0); \, }

\renewcommand{\mapsto}{ \, \tikz[baseline=-.6ex] \draw[|->,line width=.5] (0,0) -- +(.5,0); \, }

\newcommand{\A}{\mathbb{A}}
\newcommand{\F}{\mathbb{F}}

\let \P \relax
\newcommand{\P}{\mathbb{P}}

\newcommand{\Z}{\mathbf{Z}}
\newcommand{\N}{\mathbf{N}}

\newcommand{\R}{\mathbf{R}}

\newcommand{\Cl}{\mathcal C} 
\newcommand{\Dl}{\mathcal D}

\newcommand{\Gl}{\mathcal G}

\newcommand{\Ll}{\mathcal L}

\newcommand{\Tl}{\mathcal T} 
\newcommand{\Vl}{\mathcal V}

\DeclareMathOperator{\GL}{GL}
\DeclareMathOperator{\SL}{SL}
\DeclareMathOperator{\SO}{SO}

\renewcommand{\setminus}{\smallsetminus}

\renewcommand{\phi}{\varphi}
\renewcommand{\le}{\leqslant}
\renewcommand{\ge}{\geqslant}
\renewcommand{\bar}{\widebar}
\newcommand{\id}{\textup{id}}

\newcommand{\Ap}{\mathbf{E}}

\DeclareMathOperator{\Aut}{Aut}
\DeclareMathOperator{\SAut}{SAut}

\DeclareMathOperator{\End}{End}
\DeclareMathOperator{\Tame}{Tame}
\DeclareMathOperator{\STame}{STame}
\DeclareMathOperator{\Fix}{Fix}

\DeclareMathOperator{\Supp}{Supp}

\DeclareMathOperator{\car}{char}
\DeclareMathOperator{\CAT}{CAT}

\DeclareMathOperator{\Isom}{Isom}

\DeclareMathOperator{\Stab}{Stab}
\DeclareMathOperator{\inn}{in}
\DeclareMathOperator{\out}{out}
\newcommand{\trl}{L} % translation length

\newcommand{\bord}{\partial}

\newcommand{\mat}[1]{
\begin{pmatrix}#1\end{pmatrix}
}
\newcommand{\smallmat}[1]{
\left(\begin{smallmatrix}#1\end{smallmatrix}\right)
}
\newcommand{\smallpmat}[1]{
\left[\begin{smallmatrix} #1 \end{smallmatrix}\right]
}

\newcommand{\abs}[1]{\lvert #1 \rvert} %absolute value
\newcommand{\lb}{\langle}
\newcommand{\rb}{\rangle}
\newcommand{\lbb}{\mathopen{\lb\!\lb}}
\newcommand{\rbb}{\mathclose{\rb\!\rb}}
\newcommand{\normal}[2]{\lbb #1 \rbb_{#2}} % normal subgroup generated by f in G
\newcommand{\parent}[1]{\textup{(}#1\textup{)}}

\newcommand{\pcoor}[1]{%
  \begingroup\lccode`~=`: \lowercase{\endgroup
  \edef~}{\mathbin{\mathchar\the\mathcode`:}\nobreak}%
  \left[% opening symbol
  \begingroup
  \mathcode`:=\string"8000
  #1%
  \endgroup 
  \right]% closing symbol
} % projective coordinates with correct spacing: all : are treated as \mathbin{:}

\makeatletter
\def\@cite#1#2{\textup{[{#1\if@tempswa , #2\fi}]}}
\makeatother

\usepackage[%
  pdfauthor={S. Lamy},%
  colorlinks=true,%
  allcolors=Navy%
]{hyperref}

\usepackage[all]{hypcap} % to help hyperlinks direct correctly, especially for figures
\usepackage[depth=2, open, openlevel=0, numbered]{bookmark} % pour le pdf
\usepackage[capitalize,noabbrev]{cleveref}
\newcommand{\parref}[1]{\textup{\S\,}\ref{#1}}
\usepackage{upref} % for automatic \textup{\ref{...}}
\crefdefaultlabelformat{#2\textup{#1}#3}% the same for \cref

\theoremstyle{plain}
\newtheorem{theorem}{Theorem}[section]
\newtheorem{corollary}[theorem]{Corollary}
\newtheorem{proposition}[theorem]{Proposition}
\newtheorem{lemma}[theorem]{Lemma}

\theoremstyle{definition}
\newtheorem{question}{Question}

\newtheorem{example}[theorem]{Example}
\newtheorem{remark}[theorem]{Remark}

\numberwithin{equation}{section}

\NewCommandCopy\polhk\k %polish hook
\renewcommand{\k}{\mathbf{k}}
\let \O \relax
\DeclareMathOperator{\O}{O}

\title{Tame polynomial automorphisms}
\author{Stéphane Lamy}
\email{slamy@math.univ-toulouse.fr}
\address{Stéphane Lamy, Institut de Mathématiques de Toulouse,
Université de Toulouse,
118 route de Narbonne,
F-31062 Toulouse Cedex 9}
\thanks{This survey was written during a one-year stay of the author at the University of Neuchâtel, supported by the Labex CIMI}

\begin{document}

\begin{abstract}
The group $\Aut(\A^n)$ of polynomial automorphisms of the affine space is an interesting large group.
A slightly simpler group is its subgroup $\Tame(\A^n)$ of tame automorphisms.
Natural problems about these groups include the existence of normal subgroups, the classification of finite subgroups, the Tits alternative, and the possible dynamical degrees of their elements.
One method to investigate these questions is via some actions on some metric spaces, namely the coset complex and the valuation complex, that we introduce in detail.
This paper is a survey focusing on the following three cases:
the group $\Aut(\A^2) = \Tame(\A^2)$, the group $\Tame(\A^3)$, and the group $\Tame_q(\A^4)$ of tame automorphisms of $\A^4$ preserving a nondegenerate quadratic form.
\end{abstract}

\maketitle

\section{Introduction}

Let $\k$ be any field, $n \ge 1$, and $\A^n$ the $n$-dimensional affine space over $\k$.
The group $\Aut(\A^n)$ of polynomial automorphisms consists of morphisms of the form
\[
\begin{tikzcd}[map]
f\colon & \A^n & \to & \A^n \\
& (x_1, \dots, x_n) & \mapsto & (f_1, \dots, f_n),
\end{tikzcd}
\]
where $f_i \in \k[x_1, \dots, x_n]$, and admitting an inverse of the same form.
We denote such a polynomial automorphism simply by $f = (f_1, \dots, f_n)$.

Equivalently, one can think of a polynomial automorphism $f = (f_1, \dots, f_n)$ as a $\k$-automorphism of the algebra $\k[x_1, \dots, x_n]$, by sending each variable $x_i$ to $f_i$.
Beware however that the two groups $\Aut(\A^n)$ and $\Aut_\k(\k[x_1, \dots, x_n])$ are anti-isomorphic: the order of composition is reversed.
In this text we stick to the ``dynamical'' $\Aut(\A^n)$ point of view.

\begin{example}
\begin{enumerate}
\item The map $f = (x_1, x_1x_2)$ admits $(x_1, \tfrac{x_2}{x_1})$ as a rational inverse. So $f$ is a polynomial morphism from $\A^2$ to $\A^2$, but is not a polynomial automorphism.
\item
Let $\k = \R$ be the field of real numbers.
The map $x \mapsto x^3$ is polynomial, bijective from $\R$ to $\R$, but is not a polynomial automorphism of $\A^1$ since it does not admit a polynomial inverse.
\item
Let $\k = \F_q$ be a finite field.
Then a polynomial automorphism $f \in \Aut(\A^n)$ is not characterized by the bijection it induces on $\k^n$.
For instance over the field $\F_2$ with two elements, the polynomial automorphism $f = (x_1 + x_2(x_2 + 1), x_2)$ is distinct from the identity automorphism, but acts as the identity on the four points of $(\F_2)^2$.
\end{enumerate}
\end{example}

The group $\Aut(\A^n)$ contains the subgroup $\GL_n(\k)$ of linear automorphisms, and the group
\[
E_n = \left\{ (x_1 + P(x_2, \dots, x_n), x_2, \dots, x_n) \mid P \in
\k[x_2, \dots, x_n] \right\}
\]
of elementary automorphisms.
We define the tame automorphism group as the subgroup generated by linear and elementary automorphisms:
\[
\Tame(\A^n) = \langle \GL_n(\k), E_n \rangle \subset \Aut(\A^n).
\]
In particular $\Tame(\A^n)$ contains all translations, and so also the affine group
\[
A_n = \GL_n(\k) \ltimes \k^n.
\]
Conjugating elementary automorphisms by permutation matrices, we can also obtain the triangular group
\[
B_n = \{ (f_1, \dots, f_n) \mid f_i = a_i x_i + P_i(x_{i+1}, \dots, x_n) \text{ for each } i = 1, \dots, n\}.
\]
One can equivalently define the tame group as $\Tame(\A^n) = \langle A_n, B_n \rangle$.
When $n = 1$, we have
\[
\Aut(\A^1) = \Tame(\A^1) = \{ a x_1 + b \mid a \in \k^*, b\in \k \}.
\]
So this case is mostly trivial, and in this text we focus on the case $n \ge 2$.

In dimension $n = 2$, it turns out that all polynomial automorphisms are tame.
This is known as Jung's theorem, see \cite{Jung} or \cite[Chapter 7]{Lamy_Book}.
More precisely, $\Aut(\A^2)$ is the amalgamated product of the affine group $A_2$ and the triangular group $B_2$ along their intersection:

\begin{theorem}[Jung]
\label{t:jung}
Let $\k$ be any field. Then
\[
\Aut(\A^2) = \Tame(\A^2) = A_2 *_{A_2 \cap B_2} B_2.
\]
\end{theorem}

When $n = 3$, Nagata conjectured in \cite{Nagata} that the following explicit polynomial automorphism is not tame:
\begin{equation}
\label{eq:nagata}
f = (x_1 + 2x_2 (x_2^2 - x_1x_3) + x_3 (x_2^2 - x_1x_3)^2,
x_2 + x_3(x_2^2 - x_1x_3),
x_3).
\end{equation}
Observe that $f$ preserves the quadratic form $q(x_1,x_2,x_3) = x_2^2 - x_1x_3$, in the sense that $q \circ f = q$.
Moreover, $f$ is part of an additive flow.
If we write for $t \in \k$
\[
f_t = (x_1 + 2tx_2 (x_2^2 - x_1x_3) + t^2 x_3 (x_2^2 - x_1x_3)^2,
x_2 + t x_3(x_2^2 - x_1x_3),
x_3),
\]
then $f_0=\id$ and $f_{t_2} \circ f_{t_1} = f_{t_1 + t_2}$ for any $t_1, t_2 \in \k$, and in particular the Nagata automorphism $f = f_1$ admits $f_{-1}$ as an inverse.
The Nagata conjecture was established over a field of characteristic zero by Shestakov and Umirbaev \cite{ShestakovUmirbaev}, see also \cite{Kuroda,Lamy2019}:

\begin{theorem}
\label{t:SU}
Let $\k$ be a field of characteristic zero.
Then there is an algorithm to decide whether an automorphism $g \in \Aut(\A^3)$ is tame, and in particular the Nagata automorphism is not tame.
\end{theorem}

We give some detail about the algorithm in \parref{sec:results}.
Whether $\Tame(\A^n)$ is a strict subgroup of $\Aut(\A^n)$ remains open in dimension $n \ge 4$, and also in dimension $n = 3$ over a field of positive characteristic.

This text surveys some group-theoretical questions concerning the tame group.
A first question is about the existence of normal subgroups.
Observe that the Jacobian determinant gives a natural group homomorphism
\[
\begin{tikzcd}[map]
\Aut(\A^n) & \to & \k^* \\
f = (f_1, \dots, f_n) & \mapsto & \det\left( \frac{\partial f_i}{\partial x_j}\right)_{1 \le i,j \le n}.
\end{tikzcd}
\]
The special automorphism group $\SAut(\A^n)$ is the  kernel of this homomorphism, and similarly we denote by $\STame(\A^n)$ the subgroup of tame automorphisms with Jacobian determinant $1$.

\begin{question}[Simplicity problem]
\label{q:simplicity}
Does $\STame(\A^n)$ admit any nontrivial normal subgroup?
\end{question}

We will discuss how small cancellation theory allows giving a positive answer when $n = 2$ or $3$, and also for a particular subgroup of $\STame(\A^4)$. Other cases are open.
Precisely, the small cancellation theorem is \cref{t:small_cancellation}, and the application to the above listed groups are \cref{p:WPD_example,p:WPD_in_A3,p:SQuniversal}.

Now we turn to the question of classifying finite subgroups:

\begin{question}[Linearization problem]
\label{q:linearization}
Is it true that any finite subgroup of $\Tame(\A^n)$ is conjugate to a subgroup of $\GL_n(\k)$?
\end{question}

Here it is known that the answer is negative for $n \ge 5$, see \parref{sec:finite_groups}, but we show that it is positive for $n =2$ or $3$, see \cref{p:linearizable_A2,p:linearizable_A3}. The case $n = 4$ remains open in general, but again we can handle the case of the orthogonal tame group in dimension 4, see \cref{p:linearizable_Tameq}.

Another way to compare polynomial automorphisms with linear ones is via the following:

\begin{question}[Tits alternative]
\label{q:tits}
Let $G \subset \Tame(\A^n)$ be a subgroup.
Is it true that either $G$ is virtually solvable (i.e., contains a solvable subgroup of finite index), or $G$ contains a free group over two generators?
\end{question}
This alternative was established by Tits \cite{Tits72} for any linear group $\GL_n(\k)$, with $\car(\k) = 0$. (When $\car(\k) > 0$, there is also a version of the Tits alternative, by restricting to finitely generated subgroups of $G$.)
We obtain a similar alternative for $\Tame(\A^2)$ in \cref{p:tits_A2}, for $\Tame_q(\A^4)$ in \cref{p:tits_A4q}, and for $\Tame(\A^3)$ in \parref{sec:tits_A3}.

Given $f = (f_1, \dots, f_n) \in \Aut(\A^n)$, we define the ordinary degree of $f$ as $\deg(f) = \max \{ \deg(f_i) \mid i = 1, \dots, n\}$.
Then the dynamical degree of $f$ is defined as the limit
\[
\lambda(f) = \lim_{k \to \infty} (\deg(f^k))^{1/k}.
\]
The fact that the submultiplicativity of the degree implies that the limit does exist is the classical Feteke's Lemma, see e.g. \cite[Lemma 4.1]{Lamy_Book}.

\begin{question}[Dynamical degrees]
\label{q:degree}
Is it true that for any $f \in \Tame(\A^n)$, the dynamical degree $\lambda(f)$ is an algebraic number of degree at most $n -1$ ?
\end{question}

In fact the same property is expected for $f \in \Aut(\A^n)$, and also for $f \in \End(\A^n)$ replacing the bound on the degree by $n$ instead of $n-1$: see \cite[Question 1.1.2]{BlancvanSanten}, \cite[Conjecture 2]{DangFavre}.
We give the proof for $\Tame(\A^2)$ in \cref{p:dynamical_degree}.
For $\Tame_q(\A^4)$ we mention in \cref{p:gap_dang} a gap property established by Dang \cite{Dang}, and we give some examples supporting the conjecture for $\Tame_q(\A^4)$ and $\Tame(\A^3)$ in \parref{sec:dyn_degree}.

Our general strategy to tackle the questions listed above is to use an action of $\Tame(\A^n)$ on some metric spaces with nonpositive curvature property.

In \cref{sec:coset}, we introduce the coset complex, which is a $(n-1)$-dimensional connected simplicial complex with a natural action of $\Tame(\A^n)$.
In the case $n = 2$, we recover the classical Bass--Serre tree associated with the amalgamated product structure of $\Aut(\A^2) = \Tame(\A^2)$, which allows answering Questions \ref{q:linearization}, \ref{q:tits} and \ref{q:degree}.
For \cref{q:simplicity} about normal subgroups, we use the general framework of small cancellation theory, which applies to any group acting by isometries on a Gromov hyperbolic space.
Beside the case of two variables, we explain how to apply this strategy to $\Tame(\A^3)$.

In \cref{sec:autqA4}, we study a subgroup of $\Tame(\A^4)$, namely the orthogonal tame group $\Tame_q(\A^4)$ preserving the nondegenerate quadratic form $q = x_1x_4-x_2x_3$.
In this context, the analog of the coset complex becomes a square complex.
We show that not only this square complex is Gromov hyperbolic but also satisfies the $\CAT(0)$ property, which is another flavor of nonpositive curvature. This allows to answer Questions \ref{q:simplicity}, \ref{q:linearization} and \ref{q:tits} for $\Tame_q(\A^4)$, and to establish a gap property for dynamical degrees in this group.

In \cref{sec:valuations}, we are left with Questions \ref{q:linearization}, \ref{q:tits} and \ref{q:degree} for $\Tame(\A^3)$.
We introduce another metric space, the valuation complex, which turns out to carry a $\CAT(0)$ metric in the case of three variables.
Since a finite group acting on a $\CAT(0)$ space always has a fixed point,
\cref{q:linearization} reduces to the study of stabilizers for this action.
Via a deeper study of loxodromic and parabolic elements, we can also manage to establish the Tits alternative for $\Tame(\A^3)$.
Finally, regarding dynamical degrees, even if we do not get a complete answer we explain how to compute interesting examples by using the action on the valuation complex.

Finally, a word about the ground field $\k$.
When defining the coset complex or the valuation complex, we use an arbitrary field $\k$, and arbitrary dimension $n \ge 2$.
The discussion in dimension 2 is mostly independent of a choice of ground field, since Jung's \cref{t:jung} is available over any field.
However, when discussing the linearization problem we restrict to characteristic zero since we want to use an averaging argument, even if we could include the case of a finite group whose order is prime with the characteristic.
For the results in dimension 3 or 4, we need to assume characteristic zero since we always rely on some version of the Shestakov--Umirbaev theory of reduction, which is still open in positive characteristic.

\section{The coset complex}
\label{sec:coset}

In this section, we describe the construction of a simplicial complex with a natural action of $\Tame(\A^n)$.
We show that it coincides with the classical Bass--Serre tree in dimension~$2$, and explains how it helps to answer \cref{q:simplicity} in dimension $3$.

\subsection{General construction}
\label{sec:Cn}

For any $1 \le r \le n$, we say that a polynomial map
\[
\begin{tikzcd}[map]
f\colon & \A^n & \to & \A^r \\
& (x_1, \dots, x_n) & \mapsto & (f_1, \dots, f_r)
\end{tikzcd}
\]
is an $r$-tuple of components if it can be extended as a tame polynomial automorphism $(f_1, \dots, f_n) \in \Tame(\A^n)$.
The affine group $A_r = \GL_r(\k) \ltimes \k^r$ acts on the set of $r$-tuples of components by post-composition, and we use square brackets to denote an orbit for this action:
\[
[f_1, \dots, f_r] = \{ a \circ (f_1, \dots, f_r) \mid a \in A_r\}.
\]
We construct a simplicial complex $\Cl_n$ by considering each class $[f_1, \dots, f_r]$ as a vertex, and by attaching an $(n-1)$-dimensional simplex on the vertices $[f_1]$, $[f_1, f_2]$, $\dots$, $[f_1, \dots, f_n]$ for every $f = (f_1, \dots, f_n) \in \Tame(\A^n)$.
We say that a vertex of the form $[f_1, \dots, f_r]$ has type $r$, and an edge between a vertex of type $r$ and a vertex of type $s$ has type $(r,s)$.
There is a natural left action of $\Tame(\A^n)$ on the complex $\Cl_n$, given by
\begin{equation}
\label{eq:action}
g \cdot [f_1, \dots, f_r] := [f_1 \circ g^{-1}, \dots, f_r \circ g^{-1}].
\end{equation}
This action is transitive on the set of vertices of a given type $r$, and also on the set of edges of a given type $(r,s)$.
The standard simplex associated with the identity automorphism, with vertices $[x_1]$, $[x_1, x_2]$, \dots, $[x_1, \dots, x_n]$, is a fundamental domain for the action.

\begin{figure}[t]
\centering
\includegraphics{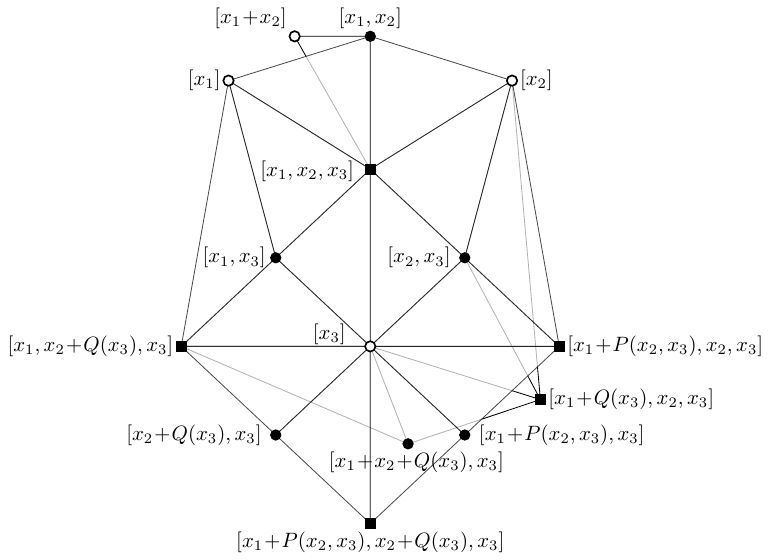}
\caption{A few simplexes in $\Cl_3$.}
\label{f:C3}
\end{figure}

\begin{example}
In dimension $n=3$ the complex $\Cl_3$ is a triangle complex, see \cref{f:C3}.
Here and in the sequel we use by convention $\circ$, $\bullet$ and \mbox{\tiny$\blacksquare$} for vertices of type $1$, $2$ and $3$ respectively.

Since $a = (x_1 + x_2, x_2) \in A_2$, we have $[x_1, x_2] = [a \circ (x_1, x_2)] = [x_1+x_2, x_2]$, and so the vertices of type 1 $[x_1]$, $[x_2]$, and $[x_1 + x_2]$ are all attached to the vertex of type $2$ $[x_1, x_2]$.
Similarly, for any $c \in \k$ the same would be true for the vertex of type $1$ $[x_1 + c x_2]$.

Given $P \in \k[x_2, x_3]$ and $Q \in \k[x_3]$, we have two ways of factorizing the triangular automorphism $(x_1 + P(x_2,x_3), x_2 + Q(x_3), x_3)$:
\begin{align*}
(x_1 + P(x_2,x_3), &\, x_2 + Q(x_3), x_3) \\
&= (x_1, x_2 + Q(x_3), x_3) \circ (x_1 + P(x_2,x_3), x_2, x_3) \\
&= (x_1 + P(x_2-Q(x_3),x_3), x_2, x_3) \circ (x_1, x_2 + Q(x_3), x_3).
\end{align*}
This produces a disk of 8 triangles around the vertex $[x_3]$.

Observe also that the relation
\begin{align*}
(x_1, x_1 + x_2 +Q(x_3), x_3) &= (-x_1 + x_2 - Q(x_3), x_2, x_3) \circ (x_2, x_1 + x_2 + Q(x_3) , x_3)
\end{align*}
implies that the two vertices of type 3 $[x_1, x_2 + Q(x_3), x_3] = [x_1, x_1 + x_2 +Q(x_3), x_3]$ and $[x_1 + Q(x_3), x_2, x_3] = [x_2, x_1 + x_2 + Q(x_3) , x_3]$ are at distance $2$, with intermediate vertex of type 2 $[x_1+x_2+Q(x_3), x_3]$.
This yields a disk of six triangles around the vertex $[x_3]$.
\end{example}

\subsection{Two variables}
\label{sec:dim2}

In dimension $n = 2$, the construction from \parref{sec:Cn} yields a graph $\Cl_2$.
In this section, we show that $\Cl_2$ is isomorphic to the classical Bass--Serre tree of $\Aut(\A^2) = \Tame(\A^2)$, as illustrated on \cref{f:tree}.
We use the following affine groups:
\begin{align*}
A_1 &= \{ax_1 + b \mid a \in \k^*, b \in \k \};\\
A_2 &= \left\lbrace (a x_1 + b x_2 + c, a' x_1 + b' x_2 + c') \mid \smallmat{a & b \\ a'&b'} \in \GL_2(\k), c,c' \in \k \right\rbrace.
\end{align*}

\begin{figure}[t]
\centering
\includegraphics{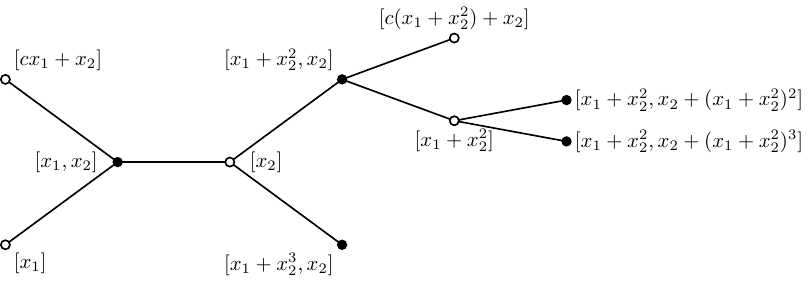}
\caption{A few vertices and edges in $\Cl_2$.}
\label{f:tree}
\end{figure}

By definition, the vertices of the graph $\Cl_2$ are of two types: classes $[f_1] = A_1 f_1$ where $f_1 \colon \A^2 \to \A^1$ is a component of an automorphism, and classes $[f_1, f_2] = A_2 (f_1,f_2)$ where $(f_1,f_2) \in \Aut(\A^2)$.
For each automorphism $(f_1,f_2) \in \Aut(\A^2)$, we attach an edge between $A_1 f_1$ and $A_2 (f_1,f_2)$.
Note that $A_2 (f_1,f_2) = A_2 (f_2,f_1)$, so there is also an edge between the vertices $A_2 (f_1,f_2)$ and $A_1 f_2$.

The edge $[x_1,x_2] - [x_2]$ is a fundamental domain for the action of $\Aut(\A^2) = \Tame(\A^2)$ on $\Cl_2$.
The stabilizer of the vertex $[x_1,x_2]$ is the affine group $A_2 = \GL_2(\k) \ltimes \k^2$, and the stabilizer of the vertex $[x_2]$ is the triangular group
\[
B_2 = \left\{ (ax_1 + P(x_2), bx_2+c) \mid a, b \in \k^*, c\in \k, P \in \k[x_2] \right\}.
\]
The fact that the graph does not admit a loop follows from the property:

\begin{lemma}
\label{l:degree}
\cite[Lemma 7.24]{Lamy_Book}
Let $f = a_0 b_1 a_1 b_2 \dots b_r a_r$ be a composition of affine and triangular automorphisms, with $b_i \in B_2 \setminus A_2$ for each $i = 1, \dots r$, and $a_i \in A_2 \setminus B_2$ for each $i = 1, \dots, r-1$.
Then $\deg(f) = \prod_{i = 1}^r \deg(b_i).$
\end{lemma}

\cref{l:degree} is a straightforward computation, and gives as an immediate corollary that $\Tame(\A^2)$ is the amalgamated product of $A_2$ and $B_2$ along their intersection, as stated in \cref{t:jung}.
Observe that the amalgamated product structure is an easy fact, the difficult part of Jung's Theorem being the equality $\Tame(\A^2) = \Aut(\A^2)$.

The fact that $\Aut(\A^2)$ acts on the tree $\Cl_2$ with fundamental domain a single edge and stabilizers of vertices $A_2$ and $B_2$, formally implies that $\Cl_2$ is the Bass--Serre tree of the amalgam.
It is not difficult to make this isomorphism explicit.
Recall that the Bass--Serre tree associated with the structure of amalgamated product $\Aut(\A^2) = A_2 *_{A_2\cap B_2} B_2$ consists in taking cosets $A_2 (f_1,f_2)$, $B_2 (f_1,f_2)$ as vertices, and cosets $(A_2 \cap B_2) (f_1,f_2)$ as edges (we use right cosets for consistency with the convention for $\Cl_2$, the classical construction with left cosets is similar).

\begin{proposition}
\label{p:C2_is_bass_serre}
The map $\phi$ defined on the set of vertices of the Bass--Serre tree by
\begin{align*}
A_2(f_1,f_2) &\mapsto A_2(f_1,f_2),\\
B_2(f_1,f_2) &\mapsto A_1 f_2,
\end{align*}
gives an isomorphism between the Bass--Serre tree and the graph $\Cl_2$.
\end{proposition}

\begin{proof}
Clearly $\phi$ is a combinatorial map, meaning it sends adjacent vertices to adjacent vertices.
Moreover, $\phi$ is bijective, since we can define $\phi^{-1}(A_1 f_2)$ to be $B_2(f_1,f_2)$, where $(f_1,f_2)$ is any automorphism admitting $f_2$ as a component.
Indeed, any other way to extend $f_2$ is of the form $(af_1+P(f_2),f_2)$, and so the class $B_2(f_1,f_2)$ does not depend on the extension we choose.
\end{proof}

From the action on the tree, we easily get an answer to \cref{q:linearization} and \cref{q:degree}, as we explain now.

To obtain the linearizability of finite subgroups, we use the following criterion, which relies on an averaging argument.
Here we take advantage of the vector structure of $\k^n$ to define an endomorphism by average of automorphisms (in general such an average has no reason to be invertible).

\begin{lemma}
\label{l:abstract_linearization}
\cite[Lemma 5.1]{BisiFurterLamy}
Let $\k$ be a field of characteristic zero, and $G$ a subgroup of the group of bijections of $\k^n$ admitting a semi-direct product structure $G = M \rtimes L$ with $L\subseteq\GL_n(\k)$.
Suppose that for any finite sequence $m_1, \dots, m_r$ of $M$, the mean $\frac1r \sum_{i=1}^r m_i$ is still in $M$.
Then any finite subgroup of $G$ is conjugate by an element of $M$ to a subgroup of $L$.
\end{lemma}

\begin{proposition}
\label{p:linearizable_A2}
\cite[Proposition 7.36]{Lamy_Book}
Let $F \subseteq \Aut(\A^2)$ be a finite subgroup.
Assume the ground field $\k$ has characteristic zero.
Then $F$ is linearizable, that is, there exists $\phi \in \Aut(\A^2)$ such
that
\[
\phi F \phi^{-1} \subseteq \GL_2(\k).
\]
\end{proposition}

\begin{proof}
A finite group acting on a tree admits a global fixed point: this is a consequence of the existence of a circumcenter for a bounded set, see \cref{l:circumcenter} for a generalization.

Then up to conjugacy we can assume that $F$ is contained in $A_2$ or $B_2$, and we conclude by applying \cref{l:abstract_linearization} to the following semidirect product structures:
\begin{align*}
A_2 &= \k^2 \rtimes \GL_2(\k),\\
B_2 &= \{ (x_1 + P(x_2), x_2 + c) \mid P \in \k[x_2], c \in \k\} \rtimes \{(ax_1, bx_2) \mid a,b \in \k^*\}.
\qedhere
\end{align*}
\end{proof}

\begin{proposition}
\label{p:dynamical_degree}
\cite[Proposition 7.25]{Lamy_Book}
For any $f \in \Aut(\A^2)$, the dynamical degree $\lambda(f)$ is an integer.
Moreover, we have the alternative:
\begin{enumerate}
\item
$\lambda(f) = 1$ if and only if $f$ is conjugate to an element in $A_2$ or $B_2$.
\item
$\lambda(f) > 1$ if and only if $f$ is conjugate to an automorphism $g$ of even length in the amalgamated product, and then $\lambda(f) = \lambda(g) = \deg (g)$.
\end{enumerate}
\end{proposition}

For the Tits alternative, the starting point is that the pointwise stabilizer of the segment $[x_1] - [x_1,x_2] - [x_2]$ is the metabelian group (recall that a metabelian group is a solvable group of derived length at most 2):
\[
\{ (ax_1 + b, cx_2 + d) \mid a, c \in \k^*, b, d \in \k \}.
\]
Then we get the following classification of subgroups:

\begin{theorem}
\label{t:subgroups_Aut_A2}
\cite[Theorem 7.61]{Lamy_Book}
Let $\k$ be any field, and $G \subset \Aut(\A^2)$ a subgroup.
We have the following classification according to the type of $G$ with respect to the action on the Bass--Serre tree:
\begin{enumerate}
\item
If $G$ is elliptic, then $G$ is conjugate to a subgroup of the affine group $A_2$ or the solvable triangular group $B_2$.
\item \label{t:subgroups_Aut_A2:2}
If $G$ is parabolic, then $G$ is metabelian and not finitely generated.
\item \label{t:subgroups_Aut_A2:3}
If $G$ is elementary loxodromic, then $G$ is solvable, and there is a loxodromic element $g \in G$ and a metabelian subgroup of elliptic elements of bounded degree $H \subset G$ such that either $G = \lb g \rb \ltimes H$, or $G$ contains a subgroup of index $2$ isomorphic to $\lb g \rb \ltimes H$.
\item
If $G$ is general loxodromic, then $G$ contains a free group over two generators.
\end{enumerate}
\end{theorem}

This yields the Tits alternative:

\begin{proposition}
\label{p:tits_A2}
\cite[Proposition 21.16]{Lamy_Book}
Let $G \subset \Aut(\A^2)$ be a subgroup.
If the ground field $\k$ has positive characteristic, we assume that $G$ is finitely generated.
Then either $G$ is virtually solvable, or contains a free group of rank $2$.
\end{proposition}

\subsection{Small cancellation}
\label{sec:SC}

We introduce the theory of small cancellation, which allows producing normal subgroups and so answering \cref{q:simplicity}.

Let $(H, d)$ be a metric space.
We say that $H$ is geodesic if for any $x, y \in H$ with $d(x,y) = \ell$, there is a map $\gamma\colon [0,\ell] \to H$ with $\gamma(0) = x$, $\gamma(d) = y$, and which is an isometry onto its image.
We say that the image of $\gamma$ is a geodesic segment, and we denote by $[x,y]$ such a segment (not necessarily unique).
Given $x,y,z \in H$, we call triangle a choice of three geodesic segments $[x,y] \cup [y,z] \cup [z,x]$.
For any $\delta \ge 0$, we say that the triangle is $\delta$-thin if the $\delta$-neighborhood of any two sides contains the third side.
The space $H$ is Gromov hyperbolic if there is a uniform $\delta \ge 0$ such that all triangles in $H$ are $\delta$-thin.
See \cref{f:delta_neigh}.

\begin{figure}
\centering
\includegraphics{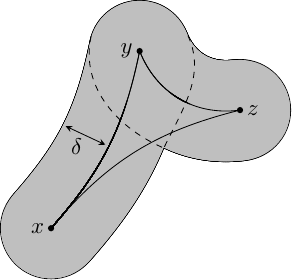}
\caption{A $\delta$-thin triangle.}
\label{f:delta_neigh}
\end{figure}

Let $f \in \Isom(H)$ be an isometry.
We say that $f$ is loxodromic if for some (hence any) $x \in H$, there exists $\lambda > 0$ such that $d(x, f^k(x)) \ge \lambda \abs k$ for any $k \in \Z$.
Now assume $G \subset \Isom(H)$ and $f \in G$ is a loxodromic element.
We say that $f$ is weakly proper discontinuous (WPD for short) if for some (hence any) $x \in H$ and any
$C \ge 1$, for $k$ sufficiently large there are only finitely many $g \in G$ satisfying $d(x, g(x)) \le C$ and $d(f^k(x), g\circ f^k(x)) \le C$.

We can now state the small cancellation theorem:

\begin{theorem}[Small Cancellation]
\label{t:small_cancellation}
\cite[Theorem 5.3, Proposition 6.34, Theorem 8.1]{DahmaniGuirardelOsin}.
Let $G$ be a group acting by isometries on a Gromov hyperbolic space $X$.
Let $g \in G$ be a WPD loxodromic element.
Then for any $C >0$, there exists an integer $n_0$ such that for any $n \ge n_0$, the normal subgroup $\normal{g^n}{G}$ generated by $g^n$ in $G$ satisfies:
\begin{enumerate}
\item There is a collection $S$ of conjugates of $g^n$ such that $\normal{g^n}{G}$ is the free group over the elements in $S$.
\item Any element $h \neq \id$ in $\normal{g^n}{G}$ is a loxodromic isometry with translation length $\trl(h) > C$.
\item In particular, if $C > \trl(g)$ then $\normal{g^n}{G}$ does not contain $g$, and so is a proper normal subgroup in $G$.
\item Moreover, $G$ is SQ-universal, meaning every countable group can be embedded into a quotient of $G$.
\end{enumerate}
\end{theorem}

A simplicial tree is a basic example of a Gromov hyperbolic space.
We can produce simple examples of WPD element in $\Aut(\A^2)$ with respect to the action on the Bass--Serre tree $\Cl_2$:

\begin{proposition}
\label{p:WPD_example}
\cite[Proposition 5.5]{LamyLonjou}, \cite[Lemma 4.23]{MO}.
\label{pro:auto WPD}

Assume $\car \k \neq 2$ \parent{resp. $\car(\k) = 2$}, and let $g = (x^2-y, x) \in \Aut(\A^2)$ \parent{resp. $g = (x^3+y, x)$}.
Then $g$ is loxodromic and satisfies the WPD property.
\end{proposition}

Combining with \cref{t:small_cancellation}, we get:

\begin{corollary}
Let $g \in \Aut(\A^2)$ be the automorphism from \cref{p:WPD_example}.
Then for any sufficiently large integer $n$, the normal subgroup generated by $g^n$ is a proper subgroup of $\Aut(\A^2)$.
\end{corollary}

There exist many criterions to establish the Gromov hyperbolicity of a given metric space.
We mention two of them, which we will use later:

\begin{proposition}[Thin bigon criterion]
\label{p:thin_bigon_criterion}
\cite{Papasoglu}
Let $\Gamma$ be a graph.
Then $\Gamma$ is Gromov hyperbolic if and only if there is a uniform constant $\eta$ such that any two combinatorial geodesics with the same endpoints must $\eta$-fellow-travel.
\end{proposition}

\begin{proposition}
\label{p:criterionhyp}
\cite[Proposition 6.1]{LamyPrzytycki2019}
Let $X$ be a simplicial complex.
Suppose that there exist constants $C,C'$, such that for each combinatorial loop $\gamma$ of length $\ell$ embedded in $X$ there is a disk diagram $D \in X$ with $\delta D = \gamma$ and a set $N \in D^0$ of cardinality $\le C \ell$ such that the $C'$-neighborhoods in $D$ around $\gamma \cup N$ cover the entire $D$. Then $X$ is Gromov hyperbolic.
\end{proposition}

\subsection{Three variables}

In this section we work over a field of characteristic zero.

\subsubsection{Results}
\label{sec:results}

\begin{proposition}
\label{p:C3-1-connected}
\cite[Proposition 5.7]{Lamy2019}
The simplicial complex $\Cl_3$ is simply connected.
\end{proposition}

\begin{corollary}
\label{c:product}
\cite[Theorem 2]{Wright}
The group $\Tame(\A^3)$ is the amalgamated product of the three following subgroups along their pairwise intersections:
\begin{align*}
\Stab([x_1, x_2, x_3]) &= A_3; \\
\Stab([x_1,x_2]) &= \{(ax_1 + bx_2 + c, a'x_1 + b'x_2 + c', \alpha x_3 + P(x_1, x_2))\}; \\
\Stab([x_1]) &= \{(ax_1 + b, f_2(x_1, x_2, x_3), f_3(x_1, x_2, x_3) \}.
\end{align*}
\end{corollary}

\cref{p:C3-1-connected} relies on the theory of elementary reductions by Shestakov and Umirbaev \cite{ShestakovUmirbaev}, and for this reason is specific to dimension 3.
Precisely, Umirbaev produces in \cite{Umirbaev} a description of relations in the group $\Tame(\A^3)$, and then Wright shows that this result can be reformulated in terms of an amalgamated product over three factors.
The theory of Shestakov and Umirbaev was reworked and somewhat simplified by Kuroda \cite{Kuroda}.
Then the paper \cite{Lamy2019} gives a description of the relations based on the work of Kuroda, by giving the prominent role to the action on the triangle complex $\Cl_3$. We now give an outline of the theory of reductions following \cite{Lamy2019}.
We do not discuss how this implies \cref{p:C3-1-connected}, but see \cref{p:C_1-connected} for a similar result in a slightly different context.

The first observation is that given a vertex $v_3$ of type $3$ in $\Cl_3$, we can choose a representative $(f_1, f_2, f_3) \in \Tame(\A^3)$ of $v_3$ such that the top monomials of the components $f_i$ are pairwise distinct.
Such a representative is not unique, but the monomials are.
This allows to define a degree function on the vertices of type 3, with value in $\N^3$.
By definition of the complex, two vertices of type 3 $v_3$ and $v_3'$ are at distance 2 in the complex if and only if (up to conjugacy by a permutation of coordinates) they admit representatives of the form $v_3 = [f_1, f_2, f_3]$ and $v'_3 = [f_1 + P(f_2, f_3), f_2, f_3]$, or in other words, representatives that differ by an elementary automorphism.
If moreover $\deg(v'_3) < \deg(v_3)$, we say that $v_3'$ is an elementary reduction from $v_3$.
Then we introduce a more involved type a reduction, called ``normal proper K-reductions''. Putting aside technical details, it consists in moving from one vertex $v_3$ to a neighbor vertex $w_3$ of the same degree, and then performing an elementary reduction from $w_3$ to another vertex $u_3$.
Then the main theorem is:

\begin{theorem}
\label{t:reducibility}
\cite[Theorem 4.1]{Lamy2019}
Any vertex of type $3$ in $\Cl_3$ admits either an elementary reduction or a normal proper K-reduction.
\end{theorem}

It is then an exercise, see \cite[Corollary 4.2]{Lamy2019}, to check that the Nagata automorphism defined in \eqref{eq:nagata} does not admit any type of reduction, and so is not tame.

We now state two further properties of the complex $\Cl_3$, whose proofs we discuss in the next subsection.

\begin{theorem}
\label{t:contractibleandgromov}
\cite[Theorems A \& B]{LamyPrzytycki2019}
The complex $\Cl_3$ is contractible and Gromov hyperbolic.
\end{theorem}

Then we exhibit loxodromic elements with the WPD property.
Note that the existence of loxodromic isometries implies that $\Cl_3$ has infinite diameter, so the property of being Gromov hyperbolic is not trivial.

\begin{proposition}
\label{p:WPD_in_A3}
\cite[Theorems C \& D]{LamyPrzytycki2019}
We consider the action of $\Tame(\A^3)$ over the graph $\Cl_3$.
Given $n \ge 0$, we write
\begin{align*}
g &= (x_2 - x_1x_3, x_1, x_3), &%& g^{-1} &= (x_2, x_1 + x_2 x_3 , x_3 ), \\
h &= (x_2, x_3, x_1),& % & h^{-1} &= (x_3, x_1, x_2 ),\\
f &= g^n \circ h.
\end{align*}
\begin{enumerate}
\item If $n \ge 3$, then $f$ is a loxodromic element.
\item If $n \ge 12$ and is even, then $f \in \STame(\A^3)$ and is a WPD element.
\end{enumerate}
\end{proposition}

\subsubsection{The curvature argument}
\label{sec:curvature}

Let $v \in \Cl_3$ be a vertex.
The link $\Ll(v)$ of $v$ is the graph of vertices at distance $1$ from $v$, with an edge between $v', v''$ if there is a triangle in $\Cl_3$ with vertices $v,v',v''$.

\begin{lemma}
Let $v \in \Cl_3$ be a vertex.
\begin{enumerate}
\item If $v$ has type $2$ then $\Ll(v)$ is a complete bipartite graph.
\item If $v$ has type $3$ then $\Ll(v)$ is isomorphic to the incidence graph of the projective plane $\P^2_\k$.
\end{enumerate}
\end{lemma}

Now we describe the link of a vertex of type $1$, which we can assume to be $[x_3]$, by transitivity of the action of $\Tame(\A^3)$.
Let $\Tl_{\k(x_3)}$ denote the Bass--Serre tree of $\Aut(\A^2)$ over the field $\k(x_3)$, see \parref{sec:dim2}.
There is a natural projection $\pi\colon \Ll([x_3]) \to \Tl_{\k(x_3)}$, which sends $[f_1, x_3]$ to $[f_1]$, and $[f_1,f_2,x_3]$ to $[f_1,f_2]$.
In particular, the fiber $\pi^{-1}([x_1])$ consists of all classes of the form $[x_1 + t(x_3), x_3]$, for $t \in \k[x_3]$ a polynomial without constant nor linear part.
Similarly, the fiber  $\pi^{-1}([x_2])$ consists of all classes of the form $[x_2 + t'(x_3), x_3]$, and they are at distance 2 from any classes in $\pi^{-1}([x_1])$, via a vertex of the form $[x_1 + t(x_3), x_2 + t'(x_3), x_3]$.
We have obtained:

\begin{lemma}
\label{l:link_type1_C3}
Let $v$ be a vertex of type $1$ in $\Cl_3$.
Then the graph $\Ll(v)$ has the following properties:
\begin{enumerate}
\item There exists an unbounded subtree $\Tl \subset \Tl_{\k(x_3)}$ and a surjective map of graphs $\pi \colon \Ll(v) \to \Tl$.
\item For any pairs $u,u'$ of vertices of type $1$ at distance $2$ in $\Tl$, and any choice of $w \in \pi^{-1}(v), w' \in \pi^{-1}(v')$, $w$ and $w'$ are at distance $2$ in $\Ll(v)$.
\end{enumerate}
\end{lemma}

We put a metric on $\Cl_3$ by considering each triangle as being isometric to the Euclidean triangle with angles $\frac{\pi}6$, $\frac{\pi}2$, $\frac{\pi}3$, at respectively the vertex of type $1$, $2$ and $3$.

Let $\Dl \subset \Cl$ be a subcomplex homeomorphic to a disk or a sphere.
Let $v \in \Dl$ be a vertex.
We denote by $d(v)$ the number of triangles in $\Dl$ around the vertex $v$, and we set $\theta(v) = \frac{\pi}6$, $\frac{\pi}2$, or $\frac{\pi}3$, according whether $v$ as type $1$, $2$ or $3$.
Then we define the curvature $K(v)$ of $\Dl$ at $v$ by the following formulas:
\[
\pi K(v) =
\begin{cases}
\pi - d(v) \theta(v) & \text{ if $\Dl$ is a disk and $v$ a boundary vertex,} \\
2\pi - d(v) \theta(v) &\text{ otherwise.}
\end{cases}
\]

\begin{proposition}[Combinatorial Gauss--Bonnet]
\label{p:gaussbonnet}
Let $\Dl \subset \Cl$ be a subcomplex.
Then
\[
\sum_{v \in \Dl} K(v) =
\begin{cases}
1 & \text{if $\Dl$ is homeomorphic to a disk},\\
2 & \text{if $\Dl$ is homeomorphic to a sphere}.
\end{cases}
\]
\end{proposition}

Now we explain how to define a corrected curvature, by adding a correcting term to the formula of $K(v)$, when two triangles of $\Dl$ meet along an edge of type $(1,2)$.

First we associate a tree to any vertex of type $2$.
We require the construction to be $\Tame(\A^3)$-equivariant, so it is sufficient to consider the vertex $v_2 = [x_1,x_2]$.
Then we consider all vertices of $\Cl_3$ of the form $[f_1(x_1,x_2)]$ or $[f_1(x_1,x_2),f_2(x_1,x_2)]$, where $(f_1, f_2) \in \Aut(\A^2)$.
These vertices of type $1$ and $2$ form a tree $\Tl(v_2)$, which is the Bass--Serre tree defined in \parref{sec:dim2}.
Observe that all vertices of $\Tl(v_2)$ are at distance $2$ from $[x_3]$, in particular the inclusion $\Tl(v_2) \subset \Cl_3$ is not an isometric embedding.

We also need the following lemma

\begin{lemma}
\label{l:orientation}
Two triangles in $\Cl$ adjacent along an edge of type $(1,2)$ can be sent by an element of $\Tame(\A^3)$ to $[x_3], [x_2,x_3], [x_1, x_2, x_3]$ and $[x_3], [x_2,x_3], [x_1 + P(x_2,x_3), x_2, x_3]$, where $P$ has degree at least $2$.
\end{lemma}

Now consider again a subcomplex $\Dl \subset \Cl_3$ homeomorphic to a disk or a sphere.
Let $T, T'$ be two triangles of $\Dl$ meeting along an edge $e$ of type $(1,2)$.
By \cref{l:orientation}, up to the action of an element $\phi \in \Tame(\A^3)$ we can assume that $T = [x_3], [x_2,x_3], [x_1, x_2, x_3]$ and $T' = [x_3], [x_2,x_3], [x_1 + P(x_2,x_3), x_2, x_3]$, where $P \in \k[x_2,x_3]$.
If $P$ depends only on one variable, meaning we can write $P(x_2,x_3) = A(f(x_2,x_3))$ with $f(x_2,x_3)$ a component of an automorphism in two variables, then we say that the edge $e$ is oriented.
Precisely, we write
\[
P(x_2,x_3) = c(a x_2 + bx_3)^d + \text{ lower order terms},
\]
and we put an arrow between $[x_3]$ and $[x_2,x_3]$ in the direction of the vertex $[a x_2 + bx_3]$ in the tree $\Tl([x_2,x_3])$.
The orientation on the initial edge $e$ is obtained by pulling back this orientation by the automorphism $\phi$.

Then given a vertex $v \in \Dl$ (of type $1$ or $2$), we denote by $\inn(v)$ (resp. $\out(v)$) the number of oriented edges to $v$ (resp. from $v$).
The corrected curvature $\tilde K(v)$ of $\Dl$ at $v$ is defined by
\[
\tilde K(v) = K(v) + \frac{\out(v) - \inn(v)}6.
\]
Since each oriented edge contributes by one to each sum $\sum_{v\in \Dl} \inn(v)$ and $\sum_{v\in \Dl} \out(v)$, the combinatorial Gauss--Bonnet formula from \cref{p:gaussbonnet} remains true if we replace $K(v)$ by $\tilde K(v)$.
The advantage of the change comes from the following fact:

\begin{proposition}
\label{p:corrected_K}
For each interior vertex $v \in \Dl$, $\tilde K(v) \le 0$.
\end{proposition}

It is important to realize that the correcting terms in the definition of $\tilde K(v)$ depend on the ambient $\Dl$ and not only on the vertex $v$. In particular, \cref{p:corrected_K} does not mean that we have a globally defined metric on $\Cl_3$ with nonpositive curvature.
However, disks and spheres in $\Cl_3$ behave as if there was such a metric.
In particular, since all vertices of a combinatorial sphere are interior, \cref{p:corrected_K} together with the combinatorial Gauss--Bonnet formula implies that there is no embedded sphere in $\Cl_3$.
This implies that $\Cl_3$ is contractible, which is the first part of \cref{t:contractibleandgromov}.

Now we turn to the study of embedded disks.

\begin{lemma}
\label{l:embedded_disk}
Let $\Dl \subset \Cl_3$ be an embedded disk, and denote by $N$ the set of interior vertices of $\Dl$ with strictly negative corrected curvature.
Then:
\begin{enumerate}
\item For each $v \in N$, $\tilde K(v) \le -1/6$.
\item For each boundary vertex $v \in \bord \Dl$, $\tilde K(v) \le 1/2$.
\item Each vertex $v \in \Dl$ is at distance at most $10$ from $\bord D \cup N$.
\end{enumerate}
\end{lemma}

This allows to apply \cref{p:criterionhyp}, and give the hyperbolicity of $\Cl_3$, which is the second part of \cref{t:contractibleandgromov}.

The proof of \cref{p:WPD_in_A3} also follows from the study of embedded disks.
From the definition of $f$, we get that $f$ sends the vertex $[x_1]$ on $[x_3]$. So by setting $\gamma(2k) = f^{k}([x_1])$ and $\gamma(2k+1) = f^{k} ([x_1, x_3])$, we get a combinatorial path in $\Cl_3$.
To prove that $f$ is loxodromic, it is sufficient to show that there exists a constant $\lambda >0$ such that $d(\gamma(s), \gamma(t)) \ge \lambda |s-t|$ for all integers $s,t$. It turns out that the constant $\lambda = \frac16$ works.
Assume by contradiction that $d(\gamma(s), \gamma(t)) < \frac16  |s-t|$.
Then assuming $|s-t|$ minimal for this property, and closing the path from $\gamma(s)$ to $\gamma(t)$ by a geodesic path, we get a loop bounding an embedded disk.
The corrected curvature at each interior vertex is nonpositive, and by the choice of $f$ the sum of the boundary curvature along the long path $\gamma$ is very negative, and cannot be cancelled out by the short geodesic segment. This contradicts the fact that the local corrected curvature contributions should add up to $2$, see \cref{p:gaussbonnet}.

Similarly, a key lemma to show that $f$ satisfies the WPD property is that given any constant $C >0$, if $t, s$ are integers sufficiently far away from each other ($|t - s| > 12 C$ is the precise condition), then for any $v, v'$ at distance at most $C$ from $\gamma(s)$ and $\gamma(t)$ respectively, any geodesic from $v$ to $v'$ meet $\gamma$.
Again the proof is by contradiction: if not, by joining $\gamma(s)$, $v$, $v'$ and $\gamma(t)$ by three geodesic paths, and using $\gamma$ as the fourth side, we obtain an embedded disk that contradicts Gauss--Bonnet formula. See \cite[Lemma 7.9]{LamyPrzytycki2019} for details.

\section{Interlude: orthogonal tame group in dimension 4}
\label{sec:autqA4}

In this section, we consider an orthogonal tame group in dimension 4, which complexity wise sits somewhere between the cases of $\Tame(\A^2)$ and $\Tame(\A^3)$.
In particular, we shall see that the natural variant of the complex $\Cl_n$ adapted to this context gives a space which is Gromov hyperbolic, but also has the $\CAT(0)$ property.

\subsection{Definition}

We work over an algebraically closed field $\k$ of characteristic zero, with the quadratic form
\[
q(x_1, x_2, x_2, x_4) = x_1x_4 - x_2 x_3,
\]
and we consider the group
\[
\Aut_q(\A^4) = \{ f \in \Aut(\A^4) \mid q \circ f = q\}.
\]
We identify $\A^4$ with the space of $2 \times 2$ matrices, so that $q$ corresponds to the determinant, and we write $f = \smallmat{f_1 & f_2 \\ f_3 & f_4}$ for an element of $\Aut_q(\A^4)$.
Observe that given three among the components $f_i$, the fourth one is uniquely determined by the condition
\[
f_1 f_4 - f_2 f_3 = x_1x_4 - x_2 x_3.
\]

The group $\Aut_q(\A^4)$ contains the linear orthogonal group $\O_4(q) = \Z/2 \ltimes \SO_4(q)$, where the $\Z/2$ factor corresponds to the transpose automorphism.
Moreover, the group $\SO_4(q)$ can be explicitly described via the 2 to 1 cover from $\SL_2(\k) \times \SL_2(\k)$ given by
\[
\begin{tikzcd}[map]
\SL_2(\k) \times \SL_2(\k) & \to & \SO_4(q) \\
(A, B) & \mapsto & {A \cdot \mat{x_1 & x_2 \\ x_3 & x_4} \cdot B^t}.
\end{tikzcd}
\]
We denote by $V_4$ the following subgroup of $\O_4(q)$:
\[
V_4 = \left\{
\id,
\mat{x_4 &x_2 \\ x_3 & x_1},
\mat{x_1 &x_3 \\ x_2 & x_4},
\mat{x_4 &x_3 \\ x_2 & x_1}
\right\} \simeq \Z/2 \times \Z/2.
\]
The group $\Aut_q(\A^4)$ also contains the subgroup of elementary automorphisms
\[
E = \left\{ \mat{x_1 + x_2 P(x_2, x_4) & x_2 \\ x_3 + x_4 P(x_2, x_4) & x_4} \mid P \in \k[x_2, x_4]\right\}.
\]
We define the tame subgroup of $\Aut_q(\A^4)$ as the group generated by linear and elementary automorphisms:
\[
\Tame_q(\A^4) = \langle \O_4(q), E \rangle.
\]

Each element of $\Aut_q(\A^4)$ restricts as an automorphism on each level set of the quadratic form $q$, and in particular to an automorphism of
\[
\SL_2 = \left\{ \mat{x_1 & x_2 \\ x_3 & x_4} \mid x_1x_4 - x_2x_3 = 1\right\}.
\]
Let $\Aut(\SL_2)$ denote the automorphism group of $\SL_2$ as an affine variety (and not as an algebraic group), and let $\Tame(\SL_2)$ be the image of $\Tame_q(\A^4)$ under the natural restriction map
\[
\rho\colon \Aut_q(\A^4) \to \Aut(\SL_2).
\]
It is unclear whether the homomorphism $\rho$ is surjective.
On the other hand, it is easy to see that $\rho$ is not injective:

\begin{example}
Let
\[
f = \mat{x_1 + x_2 (q-1) & x_2 \\ x_3 + x_4 (q-1) & x_4},
\]
where as above $q = x_1x_4-x_2x_3$.
Then $f$ is an element in $\Aut_q(\A^4)$ that restricts as the identity on $\SL_2$.
\end{example}

It turns out that the restriction of $\rho$ induces an isomorphism between $\Tame_q(\A^4)$ and $\Tame(\SL_2)$, see \cite[Proposition 4.19]{BisiFurterLamy}.
In this text we stick to the $\Tame_q(\A^4)$ point of view, even if in the main source \cite{BisiFurterLamy} it was the $\Tame(\SL_2)$ point of view which was used.

\subsection{The square complex}

We adapt the construction of the simplicial complex to get a square complex $\Cl$ on which acts the group $\Tame_q(\A^4)$.

Given an automorphism $f = \smallmat{f_1 & f_2 \\ f_3 &f_4} \in \Tame_q(\A^4)$, we define three classes:
\begin{align*}
[f_1] &= \k^* \cdot f_1 = \{af_1 \mid a \in \k^*\}, \\
[f_1, f_2] &= \GL_2(\k) \cdot (f_1, f_2) = \left\{ (af_1 + bf_2, cf_1 + df_2 \mid \smallmat{a&b\\c&d} \in \GL_2(\k) \right\},\\
\smallpmat{f_1 & f_2 \\ f_3 & f_4} &= \O_4(q) \cdot \smallmat{f_1 & f_2 \\ f_3 &f_4}.
\end{align*}
The set of vertices of the complex is the set of such classes, when $f$ runs over all elements in $\Tame_q(\A^4)$.
We say that vertices of the form $[f_1]$ (resp. $[f_1, f_2]$, resp. $\smallpmat{f_1 & f_2 \\ f_3 & f_4}$) are of type 1 (resp. type 2, resp. type 3).
By applying this definition to $\sigma \circ f$ where $\sigma \in V_4$, we see that $[f_2]$, $[f_3]$ and $[f_4]$ are also vertices of type $1$, and similarly $[f_1, f_3]$, $[f_3, f_4]$ and $[f_2, f_4]$ are vertices of type $2$.

Then we define two types of edges:
\begin{itemize}
\item
Edges between a vertex $[f_1]$ of type $1$ and a vertex $[f_1,f_2]$ of type $2$,
\item
Edges between a vertex $[f_1, f_2]$ and a vertex $\smallpmat{f_1 & f_2 \\ f_3 & f_4}$ of type $3$.
\end{itemize}

Finally, we glue a square on each loop of four edges associated with the four vertices $[f_1]$, $[f_1, f_2]$, $[f_1, f_3]$ and $\smallpmat{f_1 & f_2 \\ f_3 & f_4}$ from a given $\smallmat{f_1 & f_2 \\ f_3 & f_4} \in \Tame_q(\A^4)$.
Applying this definition to each $\sigma \circ f$ with $\sigma \in V_4$, we get four squares forming what we call the big square associated to $f$. See \cref{f:square}.

\begin{figure}
\centering
\includegraphics{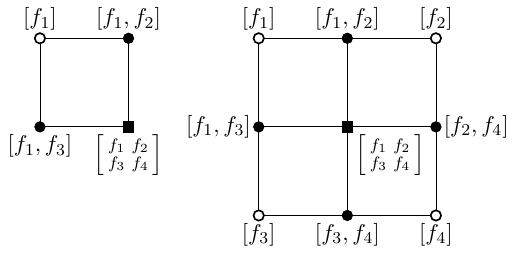}
\caption{Square and big square associated to $f \in \Tame_q(\A^4)$.}
\label{f:square}
\end{figure}

We endow the square complex $\Cl$ with the natural metric obtained by identifying each square with a Euclidean square with edges of length $1$.
The group $\Tame_q(\A^4)$ acts by isometries on $\Cl$, by the same formula as in \eqref{eq:action}.

\begin{figure}
\centering
\includegraphics{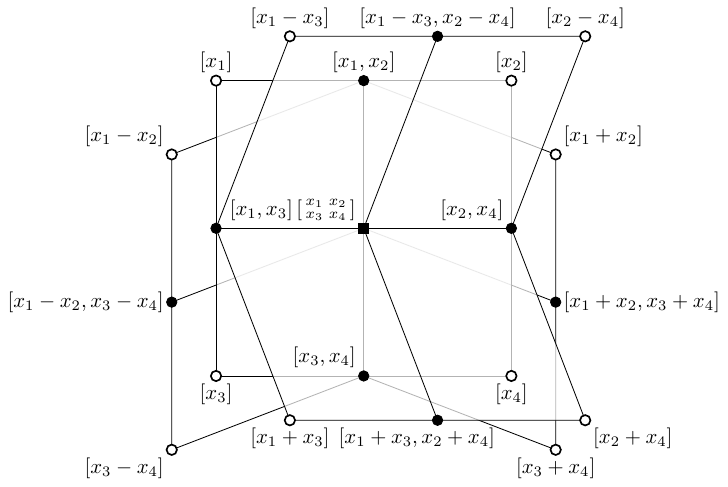}
\caption{More squares in $\Cl$.}
\label{f:more_squares}
\end{figure}

\subsection{The CAT(0) property}

Let $x, y, z \in X$ be three points in a geodesic metric space $X$.
The union $\Delta = [x,y] \cup [y,z] \cup [z,x]$ of three geodesic segments between these points is called a geodesic triangle.
We consider the Euclidean plane $\R^2$ with its standard metric, we call it the comparison space.
Given a triangle $\Delta \subset X$, we call comparison triangle a triangle $\bar \Delta \subset \R^2$ associated to a choice of three points $\bar x, \bar y, \bar z \in \R^2$ such that $d(\bar x, \bar y) = d(x,y)$, $d(\bar x, \bar z) = d(x,z)$, and $d(\bar y, \bar z) = d(y,z)$.
Such a comparison triangle $\bar \Delta$ is unique up to a Euclidean isometry.
Given $u \in \Delta$, there is a uniquely defined comparison point $\bar u \in \bar \Delta$.
Then we say that $X$ is a $\CAT(0)$ space if for any triangle $\Delta \subset X$, and any points $u,v \in \Delta$, we have $d_X(u,v) \le d_{\R^2}(\bar u, \bar v)$.
See \cref{f:CAT0}.
We say that $X$ is locally $\CAT(0)$ is for each point $x \in X$, there is $r_x >0$ such that the open ball $B(x,r_x)$ endowed with the induced metric is $\CAT(0)$.

\begin{figure}
\centering
\includegraphics{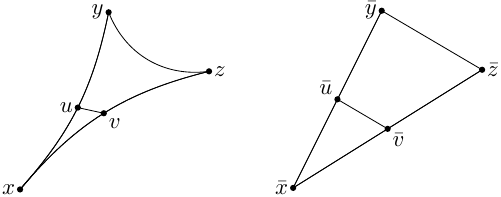}
\caption{The $\CAT(0)$ inequality.}
\label{f:CAT0}
\end{figure}

Let $X$ be a $\CAT(0)$ space, and assume moreover that $X$ is complete.
Given a bounded set $B \subset X$, the radius $r_0$ of $B$ is defined as the infimum of the positive numbers $r$ such that $B \subset B(x, r)$ for some $x \in X$.
Then there exists a unique point $x_0 \in X$, called the circumcenter of $B$, such that $B \subset B(x_0, r_0)$.
See \cite[II.2.7]{BridsonHaefliger}.

\begin{lemma}
\label{l:circumcenter}
Any finite group acting by isometries on a complete $\CAT(0)$ space fixes a point.
\end{lemma}

\begin{proof}
Consider the circumcenter of an arbitrary orbit.
\end{proof}

The standard way to check whether a metric space is $\CAT(0)$ is by using the Cartan--Hadamard's Theorem:

\begin{theorem}
\label{t:CA}
\cite[II.4.1]{BridsonHaefliger}
Let $X$ be a geodesic metric space.
Then $X$ is $\CAT(0)$ if and only if the following two properties hold:
\begin{enumerate}
\item $X$ is locally $\CAT(0)$.
\item $X$ is simply connected.
\end{enumerate}
\end{theorem}

\subsection{Geometric properties of the complex}

We aim at proving that the square complex of $\Tame_q(\A^4)$ is $\CAT(0)$.
In view of Cartan--Hadamard's \cref{t:CA}, the first step is to prove this property locally.
Since our metric space $\Cl$ is a square complex, the only place where the local curvature might not be flat is around vertices.
For instance the square complex formed by gluing three squares as in \cref{f:notCAT0} is not $\CAT(0)$.

\begin{figure}[ht]
\centering
\includegraphics{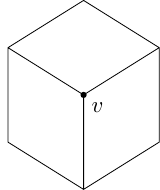}
\caption{Not locally $\CAT(0)$ at $v$.}
\label{f:notCAT0}
\end{figure}

Then the question becomes: is it true that any local loop around a vertex of $\Cl$ meets at least four squares?
The following two lemmas allows answering positively.

\begin{lemma}
\label{l:bipartite}
\cite[Propositions 3.6 \& 3.7]{BisiFurterLamy}
Let $v$ be a vertex of type $2$ or $3$ in $\Cl$.
Then the link of $v$ is a complete bipartite graph.
\end{lemma}

\begin{lemma}
\label{l:link_type1}
\cite[Lemma 3.5]{BisiFurterLamy}
Let $v$ be a vertex of type $1$ in $\Cl$, and denote by $\Gl$ its link.
Then the graph $\Gl$ has the following properties:
\begin{enumerate}
\item There exists an unbounded bicolored tree $\Tl$ and a surjective map of graphs $\pi \colon \Gl \to \Tl$.
\item For each edge $e \in \Tl$, the preimage $\pi^{-1}(e)$ is a complete bipartite graph.
\end{enumerate}
\end{lemma}

\begin{proposition}
\label{p:locally_negative}
\cite[Propositions 3.8]{BisiFurterLamy}
Let $v \in \Cl$ be any vertex.
Then any (locally injective) loop in the link of $v$ has even length at least $4$.
\end{proposition}

The second ingredient for the $\CAT(0)$ property is the global property of being simply connected.

\begin{proposition}
\label{p:C_1-connected}
\cite[Proposition 3.10]{BisiFurterLamy}
The square complex $\Cl$ is simply connected.
\end{proposition}

\begin{proof}
We consider a loop $\gamma$ in $\Cl$, and we want to find a homotopy to a constant loop. First we can assume that $\gamma$ has image in the $1$-skeleton of $\Cl$, and so is characterized by the sequence of vertices it passes through.
By \cref{l:link_type1}, the link of each vertex of type 1 is connected and so performing some local homotopy we can assume that the image of $\gamma$ contains only vertices of type $2$ and $3$, say $\gamma(2i)$ has type $3$ and $\gamma(2i+1)$ has type 2 for each $i$.
Moreover, using the action of $\Tame_q(\A^4)$ we can assume that $\gamma(0)$ corresponds to the vertex of type $3$ represented by the identity, and up to homotopy we can assume that the loop is locally injective.
Now we have a sequence of vertices $\gamma(2i)$ of type $3$. We pick the largest index $i_0$ such that $\gamma(2i_0)$ realizes the maximum of the degrees of this sequence. In particular,
\[
\deg(\gamma(2i_0-2)) \le \deg (\gamma(2i_0)) > \deg(\gamma(2i_0+2)).
\]
We can write $\gamma(2i_0) = [f]$, $\gamma(2i_0-2) = [e \circ f]$ and $\gamma(2i_0+2) = [e' \circ f]$, where $e'$ is an elementary automorphism and $e$ is conjugate to an elementary automorphism by a permutation of the coordinates.

Then the idea to finish the proof is that the existence of two distinct elementary reductions forces the existence of a reduction via a triangular automorphism, which corresponds to a disk of four squares around a vertex of type 1 and allows performing a local homotopy.
For instance \cref{f:local_homotopy} illustrates one possibility, we refer to \cite[Proposition 3.10]{BisiFurterLamy} for the detail of the three other possible cases.
\end{proof}

\begin{figure}
\centering
\includegraphics{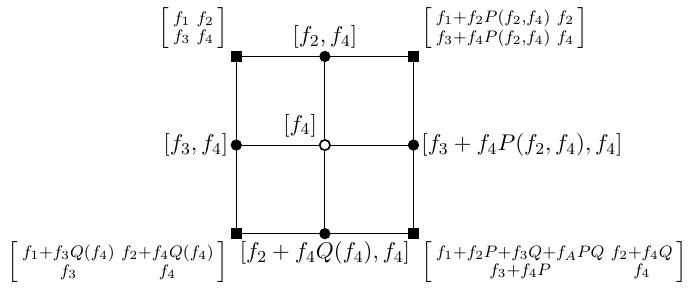}
\caption{A local homotopy.}
\label{f:local_homotopy}
\end{figure}

Putting together \cref{p:locally_negative,p:C_1-connected}, and using Cartan--Hadamard's \cref{t:CA}, we get:

\begin{theorem}
\label{t:CAT0}
\cite[Corollary 3.11]{BisiFurterLamy}
The square complex $\Cl$ is $\CAT(0)$.
\end{theorem}

By understanding the stabilizers of vertices and edges in $\Cl$ we can answer \cref{q:linearization} and \cref{q:tits}.
First using \cref{l:abstract_linearization} we get:

\begin{proposition}
\label{p:linearizable_Tameq}
\cite[Theorem B]{BisiFurterLamy}
Let $\k$ be a field of characteristic zero, and $F \in \Tame_q(\A^4)$ a finite subgroup.
Then $F$ is linearizable, meaning there exists $\phi \in \Tame_q(\A^4)$ such that $\phi F \phi^{-1} \subset \O_4(q)$.
\end{proposition}

\begin{proposition}
\label{p:tits_A4q}
\cite[Theorem C]{BisiFurterLamy}
The group $\Tame_q(\A^4)$ satisfies the Tits alternative.
\end{proposition}

To answer \cref{q:simplicity}, using the same strategy as for $\Aut(\A^2)$, we need to establish first the Gromov hyperbolicity of the complex $\Cl$. The key property is the following:

\begin{lemma}
\label{l:no_grid}
\cite[Propositions 3.13]{BisiFurterLamy}
The complex $\Cl$ does not contain any $6 \times 6$ flat grid centered on a vertex of type $1$.
\end{lemma}

Then we get:

\begin{proposition}
\label{p:hyperbolic}
\cite[Corollary 3.14]{BisiFurterLamy}
The complex $\Cl$ is Gromov hyperbolic.
\end{proposition}

\begin{proof}
Since Gromov hyperbolicity is stable under quasi-isometry, it is sufficient to prove that the 1-skeleton of $\Cl$ is Gromov hyperbolic.
Consider $x, y$ two vertices, and define the interval $[[x, y]]$ to be the union of all combinatorial geodesics from $x$ to $y$.
By \cite[Theorem 3.5]{AOS12}, $[[x, y]]$ embeds as a subcomplex of $\Z^2$.
Since by \cref{l:no_grid} there is no large flat grid in the complex $\Cl$, it follows that the 1-skeleton of $\Cl$ satisfies the criterion from \cref{p:thin_bigon_criterion}.
\end{proof}

Martin was able to produce elements in $\Tame_q(\A^4)$ satisfying a variant of the WPD property, which by \cref{t:small_cancellation} yields the following:

\begin{proposition}
\label{p:SQuniversal}
\cite[Theorem C]{Martin}
The group $\Tame_q(\A^4)$ is SQ-universal and contains free normal subgroups.
\end{proposition}

Using the complex $\Cl$, Dang studied the dynamical degrees of elements in $\Tame_q(\A^4)$, and obtained the following gap property:

\begin{proposition}
\label{p:gap_dang}
\cite[Corollary 2]{Dang}
Let $f \in \Tame_q(\A^4)$, and assume that the dynamical degree satisfies $\lambda(f) >1$.
Then $\lambda(f) \ge 4/3$.
\end{proposition}

See \cref{e:meunier} for an explicit example of automorphism $f \in \Tame_q(\A^4)$ with dynamical degree equal to the golden ratio.

\section{The valuation complex}
\label{sec:valuations}

We introduce another metric space with a natural action of $\Tame(\A^n)$.
\subsection{Construction}

A valuation on the polynomial ring $\k[x_1, \dots, x_n]$ is a function
\[
\nu\colon \k[x_1, \dots, x_n] \to \R \cup \{+\infty\}
\]
such that
\begin{itemize}
\item $\nu(P_1+P_2) \ge \min\{ \nu(P_1), \nu(P_2)\}$ for all polynomials $P_1, P_2$.
\item $\nu(P_1 P_2) = \nu(P_1) + \nu(P_2)$ for all polynomials $P_1, P_2$.
\item $\nu(P) = 0$ if and only if $P$ is a nonzero constant.
\item $\nu(P) = +\infty$ if and only if $P = 0$.
\end{itemize}
If $\nu$ is a valuation and $\lambda$ a positive real number, then $\lambda \nu$ is still a valuation.
We denote by $\Vl_n$ the space of valuations up to positive scaling.
The group $\Aut(\A^n)$ acts on valuations via the formula
\[
(g \cdot \nu)(P) = \nu(P \circ g),
\]
where $P \circ g$ is the polynomial $P(g_1, \dots, g_n)$.
This action commutes with positive scalings, and so we get an action on the space $\Vl_n$.

Now we describe some particular valuations inside $\Vl_n$.
We denote by $\Pi$ the positive quadrant inside $\R^n$:
\[
\Pi = \{(\alpha_1, \dots, \alpha_n) \mid \forall i, \alpha_i >0 \}.
\]
Let $f = (f_1, \dots, f_n) \in \Aut(\A^n)$.
Given a polynomial $P \in \k[x_1, \dots, x_n] = \k[f_1, \dots, f_n]$, we can write
\[
P = \sum_{I = (i_1, \dots, i_n)} c_I f_1^{i_1} \dots f_n^{i_n}.
\]
We call support of $P$ (with respect to $f$) the finite set of multi-indices $I$ such that $c_I \neq 0$, and we define
\[
\nu_{f^{-1}, \alpha} = \min_{I \in \Supp(P)}\left( -\sum_{k=1}^n \alpha_k i_k\right).
\]
In particular, $-\nu_{\id,\alpha}(P)$ is the usual weighted degree of $P$, where each variable $x_i$ has weight $\alpha_i$.
By our choice of notation we have
\[
f \cdot \nu_{\id, \alpha} = \nu_{f, \alpha}.
\]
Observe that for $\lambda >0$ we have $\lambda \nu_{f,\alpha} = \nu_{f,\lambda\alpha}$.
In the sequel we always silently identify valuations under scaling.
Equivalently, one could insist that all weights are normalized, for instance by one of the condition $\sum_i \alpha_i = 1$ or $\prod_i \alpha_i = 1$. The additive normalization identifies the space of weights up to scaling to the standard simplex of dimension $n-1$, this is the convention which is most convenient to represent intersection between apartments, as on \cref{f:intersections}.
The multiplicative normalization is the most natural from the point of view of the metric we will put on the valuation space, see below \parref{sec:metric}.
However, when working with rational weights, we find more convenient not to apply any of the above normalizations, and to choose the $\alpha_i$ to be coprime integers.

We say that
\[
\Ap_f =\{ \nu_{f,\alpha} \mid \alpha \in \Pi\} \subset \Vl_n
\]
is the apartment associated with $f$.
The union of all $\Ap_f$, where $f$ runs over all elements of  $\Tame(\A^n)$, is called the valuation complex, denoted by $X_n$.
Observe that here we restrict to $f \in \Tame(\A^n)$; the definition would still make sense for $f \in \Aut(\A^n)$ but would not produce a connected space.

We denote by $\Pi^+ \subset \Pi$ the subset of weights $(\alpha_1, \dots, \alpha_n)$ such that $\alpha_1 \ge \cdots \ge \alpha_n$, and $\Ap^+_f \subset \Ap_f$ the set of valuations $\nu_{f,\alpha}$ with $\alpha \in \Pi^+$.
We say that $\Ap^+_f$ is the chamber associated with $f$.
Each apartment is the union of $n!$ chambers.
Given $\alpha \in \Pi$, we write $\alpha^+$ the unique element in $\Pi^+$ obtained by reordering the $\alpha_i$.

\begin{lemma}
\label{l:fund_domain}
\cite[2.4, 2.5]{LamyPrzytycki2021}
\begin{enumerate}
\item
If $\nu_{f,\alpha} = \nu_{g,\beta}$, then $\alpha^+ = \beta^+$.
\item
The standard chamber $\Ap_{\id}^+$ is a fundamental domain for the action of $\Tame(\A^n)$ on the valuation complex.
\end{enumerate}
\end{lemma}

\subsection{Intersections of apartments}

Let $\alpha = (\alpha_1, \dots, \alpha_n) \in \Pi^+$.
We write
\[
\alpha = (\underbrace{\gamma_1, \dots, \gamma_1}_{m_1 \text{ times}}, \dots, \underbrace{\gamma_r, \dots, \gamma_r}_{m_r \text{ times}}),
\]
with $\gamma_1 > \dots > \gamma_r$, and $\sum_{i=1}^r m_i = n$.

We define
\[
L_\alpha \simeq \GL_{m_1}(\k) \times \dots \GL_{m_r}(\k) \subset \GL_n(\k) \subset \Tame(\A^n)
\]
to be the group of block diagonal matrices, with blocks of size $m_i$.
In particular, $L_\alpha$ contains the group of diagonal matrices.

For $1 \le i \le n$, we denote by $b(i)$ the integer such that $\alpha_i = \gamma_{b(i)}$.
Then we define $M_\alpha \subset \Tame(\A^n)$ as the group of triangular automorphisms of the form
\[
(x_1, \dots, x_n) \mapsto (x_1 + P_1, \dots, x_i + P_i, \dots, x_n + P_n)
\]
where each $P_i$ is a polynomial depending only on the variables $x_j$ with $b(j) > b(i)$, and with $\alpha_i \ge -\nu_{\id, \alpha}(P_i)$.
In particular, $P_n$ is always a constant, and $M_\alpha$ contains the group of translations.

\begin{proposition}
\label{p:stabilizer}
\cite[Proposition 3.2]{LamyPrzytycki2021}
Let $\alpha \in \Pi^+$.
Then the stabilizer of $\nu_{\id, \alpha}$ for the action of $\Tame(\A^n)$ on $X_n$ is the semidirect product $M_\alpha \rtimes L_\alpha$.
\end{proposition}

\begin{example}
\begin{enumerate}
\item If $\alpha = (1,\dots, 1)$ then $L_\alpha = \GL_n(\k)$ and $M_\alpha$ is the group of translations.
\item If $\alpha = (\alpha_1, \dots, \alpha_n)$ with $\alpha_1 > \dots > \alpha_n$, then $L_\alpha$ is the group of diagonal matrices and $M_\alpha$ is the group of triangular automorphisms $(x_1+P_1, \dots, x_n + P_n)$ with $\alpha_i \ge -\nu_{\id, \alpha}(P_i)$ for each $i$.
\end{enumerate}
\end{example}

\begin{proposition}
\label{p:intersection}
\cite[Corollary 2.7]{LamyPrzytycki2021}
Let $f \in \Tame(\A^n)$, and denote by $\Fix(f) \subset X_n$ the subset of valuations fixed by $f$.
Then
\[
\Ap^+_f \cap \Ap^+_\id = \Fix(f) \cap \Ap^+_\id.
\]
\end{proposition}

\begin{figure}
\centering
\includegraphics{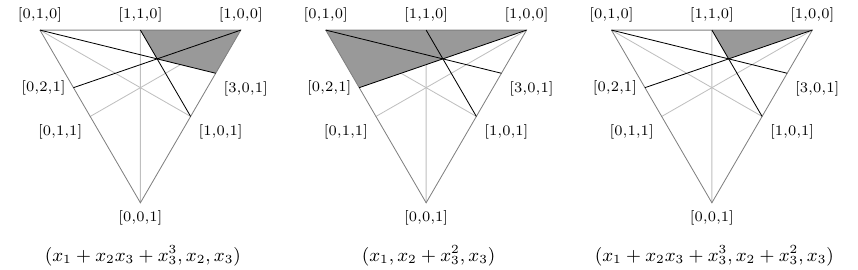}
\caption{Three examples of intersections $\Ap_\id \cap \Ap_f$.}
\label{f:intersections}
\end{figure}

Combining \cref{p:stabilizer,p:intersection} we can describe intersections of apartments.
For instance on \cref{f:intersections} the gray areas represent the intersection $\Ap_\id \cap \Ap_f$, for three examples of $f \in M_\alpha$ with $\alpha = (3,2,1)$.

\subsection{Metric}
\label{sec:metric}

Each apartment $\Ap_f$ is naturally parametrized by $\R^{n-1}$, by considering the logarithm of weights.
Precisely, to each
\[
\beta \in \left\{ (\beta_1, \dots, \beta_n) \in \R^n \mid \sum \beta_i = 0 \right\} \simeq \R^{n-1}
\]
we associate the valuation $\nu_{f,\alpha}$, where $\alpha_i = \exp(\beta_i)$.
By pushing the Euclidean metric from $\R^{n-1}$ we get a metric on each apartment $\Ap_f$.
Then we consider the induced pseudometric on the space $X_n$, defined as follows.
Given $\nu, \nu' \in X_n$, we take
\[
d(\nu, \nu') = \inf \sum d(\nu_i, \nu_{i+1}),
\]
where the infimum is taken over all sequences $(\nu_0, \dots, \nu_r)$ with $r \ge 0$, $\nu_0 = \nu$, $\nu_r = \nu'$, and $\nu_i, \nu_{i+1} \in \Ap_{f_i}$ for some $f_i \in \Tame(\A^n)$.

\begin{proposition}
\label{p:metric}
\cite[5.4 \& 5.8]{LamyPrzytycki2021}
The pseudometric $d(\cdot,\cdot)$ is a complete metric on $X_n$, and for each apartment $\Ap_f$ endowed with the Euclidean metric the inclusion $\Ap_f \subset X_n$ is an isometric embedding.
\end{proposition}

\subsection{Retraction to the coset complex}
\label{sec:retraction}

We describe a map from the valuation complex $X_n$ to the simplicial coset complex $\Cl_n$ defined in \parref{sec:Cn}.
Let $\Delta \subset \Pi^+$ be the $(n-1)$-simplex with vertices $w_1 = (2, \dots, 2, 1)$, $w_2 = (2, \dots, 2, 1,1)$, \dots, $w_n = (1,\dots, 1)$.
Given $\alpha \in \Pi^+$, we define a new weight $\alpha' \in \Delta$ by setting, for each $i = 1, \dots, n$:
\[
\alpha'_i = \min \{2, \alpha_i / \alpha_n \}.
\]
We identify $\Delta$ with the standard simplex in $\Cl_n$, by sending the weight
\[w_i = (\underbrace{2, \dots, 2}_{n-i}, \underbrace{1 , \dots 1}_i)\] to the vertex $[x_{n+1-i}, x_{n+2-i}, \dots, x_n]$.
We write $\iota\colon \Delta \to \Cl_n$ the corresponding injection.
Now let $\nu \in X_n$.
We can choose a representative $\nu = \nu_{f,\alpha}$ with $\alpha \in \Pi^+$.
Then we set $\pi(\nu) = f\cdot \iota(\alpha') \in \Cl_n$.
By \cite[Lemma 9.3]{LamyPrzytycki2021}, this yields a well-defined continuous surjective map $\pi\colon X_n \to \Cl_n$.

In the case $n = 2$, we get a map between two trees, as illustrated on \cref{f:retraction_n=2}.

\begin{figure}
\centering
\includegraphics{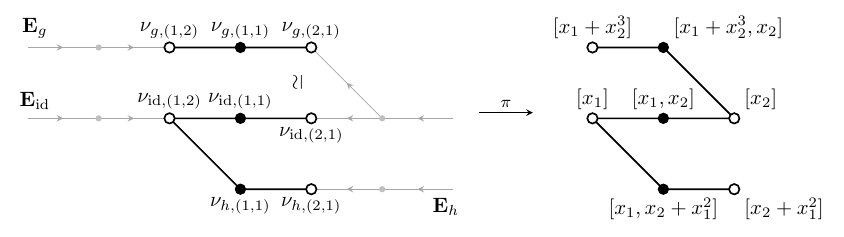}
\caption{The retraction map when $n = 2$, on the apartments associated with $\id$, $g = (x_1+x_2^3,x_2)$ and $h = (x_1, x_2+x_1^2)$.}
\label{f:retraction_n=2}
\end{figure}

On \cref{f:retraction} we illustrate the restriction map when $n = 3$.
One could expect that $\pi$ always induces a homotopy equivalence, for any number $n$ of variables.
Since we proved in \cref{p:C3-1-connected} that the coset complex $\Cl_3$ is simply connected, this would imply the same property for the valuation complex.
The details turn out to be a little more intricate, but using this idea and the amalgamated product structure from \cref{c:product} we can indeed prove:

\begin{proposition}
\label{p:X3_contractible}
\cite[Proposition 6.3]{LamyPrzytycki2021}
The space $X_3$ is simply connected.
\end{proposition}

\begin{figure}[ht]
\centering
\includegraphics{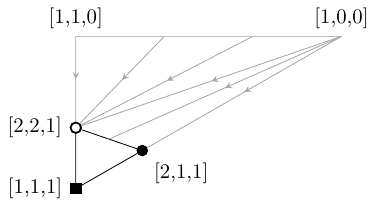}
\caption{The retraction map when $n=3$, in restriction to $\Ap_{\id}^+$.}
\label{f:retraction}
\end{figure}

\subsection{Negative curvature}
\label{sec:neg_curvature}

\begin{proposition}
\label{p:courbure_negative}
\cite[Proposition 7.1]{LamyPrzytycki2021}
Let $\alpha \in \Pi$, and $g \in \Tame(\A^3)$.
Then $\nu_{g,[\alpha]}$ admits a $\CAT(0)$ neighborhood in $X_3$.
\end{proposition}

The proof of this proposition is the most technical part of the paper \cite{LamyPrzytycki2021}.
The following example illustrates one of the main ingredient in the proof.

\begin{example}
\label{e:local_curvature}
Let $p \ge 1$ and $a \ge b$ be integers, and set $m = (a+b)p$.
The following two polynomials are homogenous of degree $m$ with respect to the variables $x_2, x_3$ with weighted degree $p$ and $1$:
\begin{align*}
P_1(x_2, x_3) &= (x_2-x_3^p)^a x_2^b, &
P_3(x_2, x_3) &= -x_2^a(x_2+x_3^p)^b.
\end{align*}
Then we consider the following elementary automorphisms:
\begin{align*}
h_0 &= (x_1, x_2-x_3^p, x_3), &
h_1 &= (x_1 + P_1(x_2, x_3), x_2, x_3), \\
h_2 &= (x_1, x_2+x_3^p,x_3), &
h_3 &= (x_1 + P_3(x_2,x_3), x_2, x_3).
\end{align*}
We compute
\begin{align*}
h_0 h_1 &= (x_1 + P_1(x_2, x_3), x_2-x_3^p, x_3) \\
h_2 h_3 &= (x_1 + P_3(x_2, x_3), x_2 + x_3^p, x_3).
\end{align*}
Since $P_1(x_2+x_3^p, x_3) = -P_3(x_2, x_3)$, we obtain
\[
h_0h_1h_2h_3 = \id.
\]
Observe that each $h_i$ fixes the valuation $\nu = \nu_{\id, [m,p,1]}$.
We consider the loop in the link of $\nu$ consisting of four arcs in the apartments $\Ap_{h_0}$, $\Ap_{h_0h_1}$, $\Ap_{h_0h_1h_2} = \Ap_{h_3^{-1}}$, $\Ap_{h_0h_1h_2h_3} = \Ap_{\id}$. This gives a total angular length $\theta_a + \theta_b + \theta_a + \theta_b$, where $\frac{\pi}3 \le \theta_a \le \theta_b \le \frac{2\pi}{3}$, see \cref{f:local_curvature}.
To get nonpositive curvature we need to show $2(\theta_a + \theta_b) \ge 2\pi$, or equivalently $(\theta_a - \frac\pi3) + (\theta_b - \frac\pi3) \ge \frac\pi3$.
In fact in this situation we have an equality, by \cref{l:angulaire} below.
\end{example}

\begin{figure}
\centering
\includegraphics{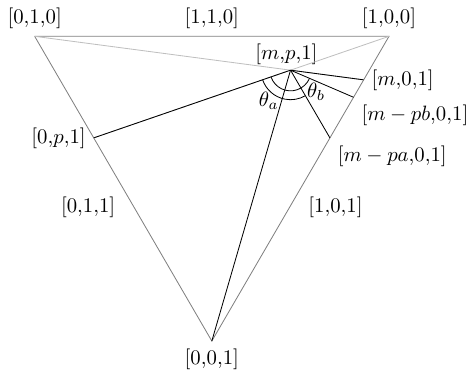}
\caption{Arc pieces in \cref{e:local_curvature}.}
\label{f:local_curvature}
\end{figure}

\begin{lemma}
\label{l:angulaire}
Let $m, p, k$ be integers with $m > p \ge 1$ and $m \ge 2k \ge 0$.
Let $c_1, c_2, c_3, c_4$ be the half-lines from $[m,p,1]$ to respectively $[0,0,1]$, $[k,0,1]$, $[m-k,0,1]$, $[m,0,1]$: see \cref{f:lemme_angulaire}.
For $1 \le i < j \le 4$, denote by $\theta_{ij}$ the angle at the point $[m,p,1]$ between the curves $c_i$ and $c_j$, for the metric $|\cdot,\cdot|$.
Then $\theta_{12} = \theta_{34}$, or equivalently $\theta_{12} + \theta_{13} = \pi/3$.
\end{lemma}

\begin{proof}
Since $[0,1,0]$ is collinear with $[m,p,1]$ and $[m,0,1]$, by definition of the metric $c_1$ and $c_4$ are still lines with respect to the metric $|\cdot,\cdot|$ and $\theta_{14}=\frac{\pi}{3}$.
The fact that the two conclusions are equivalent then comes from the equality
\[
\frac{\pi}{3} = \theta_{14}
=  \theta_{12} +  \theta_{23} + \theta_{34} = \theta_{12} + \theta_{13} + (\theta_{34} - \theta_{12}).
\]

Consider the involution on the space of weights $\Pi$ given by
\[
\tau\colon [\alpha_1, \alpha_2, \alpha_3] \mapsto [\alpha_1, p \alpha_3, \alpha_2 / p].\]
This involution fixes the line $\{ [t, p, 1] \mid t \ge 0\}$, which contains $[m,p,1]$.
Moreover, $\tau [0,0,1] = [0,1,0]$ which is collinear with $[m,p,1]$ and $[m,0,1]$, and
$\tau [k,0,1] = [k,p,0]$ which is collinear with $[m,p,1]$ and $[m-k,0,1]$.
In particular, $\tau$ exchanges $c_1$ with $c_4$, and also $c_2$ with $c_3$.

At the level of $\beta_i = \log \alpha_i$, the involution $\tau$ becomes:
\[
(\beta_1, \beta_2, \beta_3) \mapsto (\beta_1, \beta_3 +
\log p, \beta_2 - \log p).
\]
Thus, for the metric $|\cdot,\cdot|$ the involution $\tau$ is an axial symmetry, with axis $\beta_2 - \beta_3 = \log p$.
In particular, it preserves the non-oriented angles and we conclude $\theta_{12} = \theta_{34}$.
\end{proof}

\begin{figure}
\centering
\includegraphics{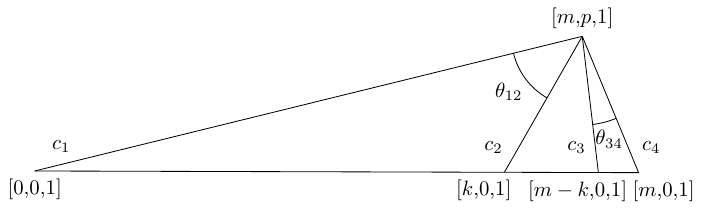}
\caption{Setting of \cref{l:angulaire}.}
\label{f:lemme_angulaire}
\end{figure}

As a consequence of Cartan--Hadamard's \cref{t:CA}, from \cref{p:X3_contractible,p:courbure_negative} we obtain:

\begin{theorem}
\label{t:X3_CAT0}
Over a field $\k$ of characteristic zero, the space $X_3$ is a $\CAT(0)$ complete metric space.
\end{theorem}

\subsection{Finite subgroups}
\label{sec:finite_groups}

We obtain a positive answer to \cref{q:linearization} in dimension $n = 3$:

\begin{proposition}
\label{p:linearizable_A3}
Let $\k$ be a field of characteristic zero, and $F \subset \Tame(\A^3)$ a finite subgroup.
Then $F$ is linearizable, meaning there exists $\phi \in \Tame(\A^3)$ such that $\phi F \phi^{-1} \subset \GL_3(\k)$.
\end{proposition}

\begin{proof}
The group $F\subset \Tame(\A^3)$ acts by isometries on $X_3$, which is a complete $\CAT(0)$ space by \cref{t:X3_CAT0}.
The group $F$ admits (at least) one global fixed point $\nu\in X_3$, which can be constructed as the “circumcenter” of an arbitrary orbit.
By \cref{l:fund_domain}, conjugating by an element of $\Tame(\A^3)$ we can assume that $\nu = \nu_{\id,[\alpha]} \in \Ap_\id^+$.
By \cref{p:stabilizer}, $F\subset \Stab(\nu) = M_\alpha \rtimes L_\alpha$.
Since the group of triangular automorphisms $M_\alpha$ is stable by mean, we can therefore conclude by \cref{l:abstract_linearization}.
\end{proof}

We mention that there are nonlinearizable finite subgroups in $\Aut(\A^4)$.
For example, following \cite{FMJ}, if $\k$ contains three cube roots of unity, $1$, $\omega$ and $\omega^{-1}$, then the action on~$\A^4$ of the symmetric group $S_3 = \langle \sigma, \tau \mid \sigma^3 = \tau^2 = (\sigma \tau)^2 = 1\rangle$ defined as follows is nonlinearizable:
\begin{align*}
\sigma(a,b,x,y) &= (\omega a, \omega^{-1} b, x,y), \\\
\tau(a,b,x,y) &= (b,a, -b^3x + (1+ab+a^2b^2)y, (1-ab)x + a^3y ).
\end{align*}
One may suspect that $\tau$ is not a tame automorphism, but this falls under the open question of whether the inclusion $\Tame(\A^n) \subset \Aut(\A^n)$ is strict in dimension $n \ge 4$.
On the other hand, as $\tau$ after composition by the involution $(a,b,x,y) \mapsto (b,a,y,x)$ is identified with
an element of $\SL_2(\k[a,b])$, it is known thanks to \cite{Suslin} that this example becomes tame after extension to $\k^5$, by trivially extending the action on the additional variable.
Indeed, Suslin proves that for any $r \ge 3$, the group $\SL_r(\k[x_1, \dots, x_n])$, which can be thought of as a subgroup of $\Aut(\k^{r + n})$, is generated by elementary matrices, which are particular tame automorphisms.
Moreover, still following \cite{FMJ}, the previous example remains nonlinear in $\Aut(\A^{4+m})$, if we trivially extend the action of $S_3$ to any number $m$ of additional variables.
Combining with the result of Suslin, for any $n \ge 5$ we get a non-linear finite subgroup of $\Tame(\A^n)$;
in other words \cref{p:linearizable_A3}, and thus also \cref{t:X3_CAT0}, are no longer valid in dimension $n \ge 5$.
However, it would be interesting to study whether one of the two ingredients of the $\CAT(0)$ property persists, namely simple connectedness or local $\CAT(0)$ property.
Finally, the case of the dimension $n = $4 remains open, but seems difficult in the absence of a theory of reductions.

\subsection{Tits alternative}
\label{sec:tits_A3}

We start with two general lemmas:

\begin{lemma}
\label{l:tits_stable_extension}
\cite[Lemma 5.5]{Dinh}, \cite[Lemma 21.5]{Lamy_Book}
Let $G$ be a group and $N \lhd G$ a normal subgroup.
Then $G$ satisfies the Tits alternative if and only if the normal subgroup $N$ and the quotient group $G/N$ also do.
\end{lemma}

\begin{lemma}
\label{l:tits_uniform_countable}
\cite[Corollary 12.3]{LamyPrzytycki2026}
Let $n, m > 0$ and let $G$ be a group.
Suppose that each finitely generated subgroup of $G$ contains a nonabelian free group or a solvable subgroup of index $\le n$ and
derived length $\le m$.
Then $G$ satisfies the Tits alternative
\end{lemma}

By \cref{p:stabilizer}, the stabilizer of a point $\nu_{f,\alpha} \in X_3$ is the semidirect product of a linear group $L_\alpha$ and a solvable group $M_\alpha$.
So by \cref{l:tits_stable_extension}, such a stabilizer satisfies the Tits alternative.
Similarly, we can establish the Tits alternative for the following two subgroups $B', C \subset \Tame(\A^3)$, which correspond to stabilizers of particular endpoints in $\bord X_3$.

The group $B'$ consists of automorphisms of the form
\[
ax_1 + P(x_2, x_3), f(x_2,x_3), g(_2,x_3) ),
\]
where $(f,g)$ is an automorphism of $\A^2$ (with variables $x_2, x_3$). Since we know from \cref{p:tits_A2} that $\Aut(\A^2)$ satisfies the Tits alternative, we obtain that $B'$ also does.

The group $C$ consists of automorphisms of the form
\[
(f(x_1, x_2, x_3), g(x_1,x_2,x_3), ax_3 + b).
\]
Here we can think of $(f,g)$ as an automorphism of $\A^2$ with variables $x_1, x_2$, and coefficients in the field of fractions of $\k[x_3]$. So again by the same argument as above we obtain that $C$ satisfies the Tits alternative.

Now we review the proof of the Tits alternative for $\Tame(\A^3)$, see \cite{LamyPrzytycki2026}.
Every element of $\Tame(\A^3)$ is elliptic, parabolic or loxodromic, according to its action on the $\CAT(0)$ space $X_3$.
Let $G \subset \Tame(\A^3)$ be a subgroup.

If $G$ contains two loxodromic elements with disjoint pairs $(\xi, \eta)$, $(\xi', \eta')$ of fixed endpoints, and satisfying the technical condition that each pair consists of far limit points, then by a standard ping-pong argument we can produce a free group over two generators in $G$.

If all elements in $G$ are elliptic, then any finitely generated subgroups $H \subset G$ admits a global fixed point, and $H$ is conjugate to a subgroup of $M_\alpha \rtimes L_\alpha$ for some $\alpha$. We are then in position to apply \cref{l:tits_uniform_countable}.

If $G$ stabilizes a principal or antiprincipal endpoint, then $G$ is conjugate to a subgroup of $B'$ or $C$ and so satisfies the Tits alternative.

The end of the proof is to show that $G$ must fall under one of the cases treated above.

\subsection{Dynamical degrees}
\label{sec:dyn_degree}

Let $\alpha = (\alpha_1, \dots, \alpha_n) \in \Pi^+$, and $f \in \End(\A^n)$.
We define the $\alpha$-degree of $f$ as
\[
\deg_{\alpha}(f) = \max \left\{ \frac{-\nu_{\id,\alpha(f_i)}}{\alpha_i} \mid i = 1, \dots, n \right\}.
\]

Given a polynomial $P \in \k[x_1,\dots, x_n]$, we can write $P = \sum_{\theta} P_\theta$, where each $P_\theta$ is a sum of monomials $m_i$ such that $-\nu_{\id,\alpha}(m_i) = \theta$.
We say that $P_\theta$ is the $\alpha$-homogeneous part of degree $\theta$ of $P$.
Similarly, given the endomorphism $f \in \End(\A^n)$, we define the $\alpha$-homogeneous part of degree $\theta$ as $f_\theta = (g_1, \dots, g_n)$, where $g_i$ is the $\alpha$-homogeneous part of degree $\theta \cdot \alpha_i$ of $f_i$.
In particular, if $\theta = \deg_{\alpha}(f)$, we say that the $\theta$-homogeneous part is the $\alpha$-leading part of $f$.
We say that $f$ is $\alpha$-algebraically stable if $\deg_\alpha(f^k) = \deg_\alpha(f)^k$ for all $k \ge 1$

\begin{proposition}
\label{p:muAS}
\cite[Proposition A]{BlancvanSanten}
Let $f \in \End(\A^n)$ be an endomorphism of dynamical degree $\lambda(f)$, $\alpha = (\alpha_1, \dots, \alpha_n) \in (\R_{>0})^n$ a weight, and $g \in \End(\A^n)$ the $\alpha$-leading part of $f$.
Then the following are equivalent:
\begin{enumerate}
\item $\deg_{\alpha}(f) = \lambda(f)$;
\item $f$ is $\alpha$-algebraically stable;
\item $g^k \neq 0$ for each $k \ge 1$.
\end{enumerate}
\end{proposition}

\begin{example}
\cite[Lemma 4.3.5]{BlancvanSanten}
Let $a,b,c \ge 1$, and consider
\[
f = (x_3 + x_1^ax_2^b, x_2 + x_1^c, x_1) \in \Tame(\A^3).
\]
Given $\alpha = (\alpha_1, \alpha_2, \alpha_3)$,
\[
\deg_\alpha(f) = \max \left\{ \frac{\alpha_3}{\alpha_1}, \frac{a\alpha_1 + b\alpha_2}{\alpha_1}, \frac{\alpha_2}{\alpha_2}, \frac{c\alpha_1}{\alpha_2}, \frac{\alpha_1}{\alpha_3}\right\}.
\]
Since the formula is invariant under scaling, we can assume $\alpha_1 = 1$.
Then let $\alpha_2$ be the unique positive solution of the equation
\[
a + b \alpha_2 = c \frac{1}{\alpha_2}, \text{ or equivalently } b \alpha^2_2 + a \alpha_2 - c = 0.
\]
So
\[
\alpha_2 = \frac{-a + \sqrt{a^2+4bc}}{2b}.
\]
Let $\lambda = a + b\alpha_2 = a + \frac12 (-a + \sqrt{a^2+4bc}) = \frac12 (a + \sqrt{a^2+4bc}) $, and pick $\alpha_3$ such that $\frac{1}{\lambda} < \alpha_3 < \lambda$.
Then by construction $\deg_\alpha(f) = \lambda$, and the $\alpha$-leading part of $f$ is $g = (x_1^ax_2^b, x_1^c, 0)$.
Since $g$ fixes $(1,1,0)$, no power of $g$ is the zero endomorphism, and we conclude by \cref{p:muAS} that $f$ is $\alpha$-algebraically stable, with dynamical degree equal to $\lambda(f) = \frac12 (a + \sqrt{a^2+4bc}) $.
\end{example}

One can similarly obtain some interesting examples of dynamical degrees for the group $\Tame_q(\A^4)$, which we considered in \parref{sec:autqA4}.

\begin{example}
\label{e:meunier}
\cite{Meunier}
Let
\[
f = \mat{x_3 & x_1 + x_3^2 \\ -x_4 & -x_2-x_3x_4 } \in \Tame_q(\A^4).
\]
Given $\alpha = (\alpha_1, \alpha_2, \alpha_3, \alpha_4)$,
\[
\deg_\alpha(f) = \max \left\{ \frac{\alpha_3}{\alpha_1}, \frac{\alpha_1}{\alpha_2}, \frac{2\alpha_3}{\alpha_2} , \frac{\alpha_4}{\alpha_3}, \frac{\alpha_2}{\alpha_4}, 1 + \frac{\alpha_3}{\alpha_4}\right\}.
\]
% In particular, if $\alpha = (\alpha_1, \alpha_2, 1-\alpha_2, 1-\alpha_1)$,
% \[
% \deg_\alpha(f) = \max \left\{ \frac{1-\alpha_2}{\alpha_1}, \frac{\alpha_1}{\alpha_2}, \frac{2-2\alpha_2}{\alpha_2} , \frac{1-\alpha_1}{1-\alpha_2}, \frac{\alpha_2}{1-\alpha_1}, 1 + \frac{1-\alpha_2}{1-\alpha_1}\right\}.
% \]
Let $c = \frac{\sqrt 5-1}{2}$, solution of $\frac1c = 1+c$, and let $\alpha = (1-\alpha_4, 1-c \alpha_4, c \alpha_4, \alpha_4)$. Then
\begin{align*}
\deg_\alpha(f) &= \max \left\{ \frac{c\alpha_4}{1-\alpha_4}, \frac{1-\alpha_4}{1-c\alpha_4}, \frac{2c\alpha_4}{1-c \alpha_4} , \frac{\alpha_4}{c\alpha_4}, \frac{1-c\alpha_4}{\alpha_4}, 1 + \frac{c\alpha_4}{\alpha_4}\right\} \\
&=
 \max \left\{ \frac{c\alpha_4}{1-\alpha_4}, \frac{1-\alpha_4}{1-c\alpha_4}, \frac{2c\alpha_4}{1-c \alpha_4} , \frac{1-c\alpha_4}{\alpha_4}, 1 + c \right\}.
\end{align*}
For $\frac{1}{5\sqrt5} < \alpha_4 < \frac{5+\sqrt 5}{10}$, one checks that $\deg_\alpha(f) = 1+ c$, and the $\alpha$-leading part of $f$ is $g = (0,0,-x_4, -x_3x_4)$.
We conclude that $f$ is $\alpha$-algebraically stable, with dynamical degree $\lambda(f) = \frac{1+\sqrt5}{2} \approx 1,618$.
\end{example}

\begin{remark}
In \cite{Meunier}, Meunier shows that for any $f \in \Tame_q(\A^4)$ with $\deg(f) = 2$, we have $\lambda(f) \in \{1, 2, \frac{1+\sqrt5}{2}\}$.
So $\frac{1+\sqrt5}{2}$ instead of $\frac 43$ seems to be a good candidate to be the optimal lower bound in \cref{p:gap_dang}.
\end{remark}

\bibliographystyle{myalpha2024}
\bibliography{biblio}

\end{document}